\documentclass[12pt,reqno]{amsart}

\usepackage{mathrsfs}
\usepackage{amsmath}
\usepackage{amssymb}
\usepackage{amsfonts}
\usepackage{amsthm}
\usepackage{a4wide}
\usepackage{graphicx}
\usepackage{color}
\usepackage[sans]{dsfont}
\usepackage{linearA}
\usepackage{pgfkeys}
\usepackage{longtable}
\usepackage{pdflscape}
\usepackage{float}
\usepackage{cancel}
\usepackage{dsfont}
\usepackage{setspace}
\usepackage{array}
\usepackage{hyperref}
\hypersetup{colorlinks=true,allcolors=[rgb]{0,0,0.6}}
\usepackage[english]{babel}
\usepackage[applemac]{inputenc}
\usepackage[T1]{fontenc}
\usepackage{tikz,pgflibraryshapes}
\usepackage{enumitem}
\usepackage{eufrak}
\usepackage[square,numbers,sort&compress]{natbib}
\usepackage[shadow,textwidth=22mm,textsize=tiny]{todonotes}
\usepackage[left=2.5cm,right=2.5cm,top=3cm,bottom=3cm]{geometry}
\usepackage{soul}

\newcommand{\N}{\mathbb{N}}

\newcommand{\R}{\mathbb{R}}

\newcommand{\Z}{\mathbb{Z}}

\newtheorem{theorem}{Theorem}[section]
\newtheorem{lemma}[theorem]{Lemma}
\newtheorem{proposition}[theorem]{Proposition}

\newtheorem{remark}[theorem]{Remark}
\newtheorem{assumption}[theorem]{Hypothesis}

\numberwithin{equation}{section}

\author[S. Breteaux]{S{\'e}bastien Breteaux}
\address[S. Breteaux]{Universit{\'e} de Lorraine, CNRS, IECL, F-57000 Metz, France}
\email{sebastien.breteaux@univ-lorraine.fr}

\author[J. Faupin]{J{\'e}r{\'e}my Faupin}
\address[J. Faupin]{Universit{\'e} de Lorraine, CNRS, IECL, F-57000 Metz, France}
\email{jeremy.faupin@univ-lorraine.fr}

\author[V. Grasselli]{Viviana Grasselli}
\address[V. Grasselli]{Universit{\'e} de Lorraine, CNRS, IECL, F-57000 Metz, France}
\email{viviana.grasselli@univ-lorraine.fr}

\theoremstyle{plain}

\theoremstyle{plain}
\theoremstyle{plain}

\theoremstyle{plain}

\theoremstyle{remark}

\makeatother

\usepackage{babel}

\providecommand{\lemmaname}{Lemma}
\providecommand{\notationname}{Notation}
\providecommand{\propositionname}{Proposition}
\providecommand{\theoremname}{Theorem}

\begin{document}
 \title[Number of bound states for fractional Schr\"odinger operators]{On the number of bound states \\ for fractional Schr\"odinger operators with critical and super-critical exponent}

\begin{abstract}
We study the number $N_{<0}(H_s)$ of negative eigenvalues, counting multiplicities, of the fractional Schr\"odinger operator $H_s=(-\Delta)^s-V(x)$ on $L^2(\mathbb{R}^d)$, for any $d\ge1$ and $s\ge d/2$. We prove a bound on $N_{<0}(H_s)$ which depends on $s-d/2$ being either an integer or not, the critical case $s=d/2$ requiring a further analysis. Our proof relies on a splitting of the Birman-Schwinger operator associated to this spectral problem into low- and high-energies parts, a projection of the low-energies part onto a suitable subspace, and, in the critical case $s=d/2$, a Cwikel-type estimate in the weak trace ideal $\mathcal{L}^{2,\infty}$ to handle the high-energies part.
\end{abstract}
\maketitle
\tableofcontents
\section{Introduction}

Estimating the number of bound states of the two-body Schr\"odinger operator
\begin{equation*}
H:= -\Delta - V(x) 
\end{equation*}
on $L^2(\mathbb{R}^d)$ constitutes a rich problem that has attracted a large amount of attention in the mathematical literature. Classical textbook references include \cite[Chapter XIII.3]{ReedSimon4}, \cite[Chapter 7]{Simon05}, \cite[Chapter XI]{EdmundEvans18}, \cite[Chapter 4]{lieb_seiringer}, see also \cite{Simon76_2,Hundertmark02} for review articles and \cite[Chapter 4]{FrankLaptevWeidl23} for a more recent exposition.

Roughly speaking, the question raised is as follows. Consider a real-valued measurable function $V:\mathbb{R}^d\to\mathbb{R}$ such that $H$ identifies with a self-adjoint operator on $L^2(\mathbb{R}^d)$, with essential spectrum $[0,\infty)$ (see e.g. \cite{ReedSimon2} or \cite{FrankLaptevWeidl23} for sufficient conditions on $V$ implying these properties, see also Hypothesis \ref{hyp:v-high-frequency} and Remark \ref{rk:embeddings} below for the conditions considered in this paper, in the setting of the fractional Schr\"odinger operator). The bound states of $H$ are defined as the normalized eigenvectors corresponding to negative eigenvalues. One then aims at estimating $N_{<0}(H)$, the number of negative eigenvalues of $H$ counting multiplicities.

Note that, decomposing $V=V_+-V_-$ with $V_\pm\geq0$, we have $H\geq -\Delta -V_+(x)$ in the sense of quadratic forms, which implies that
\[N_{<0}(H)\le N_{<0}\big(-\Delta-V_+(x)\big).\]
Therefore, to obtain a bound on the number of negative eigenvalues of $H$, it suffices to consider the case where $V=V_+\geq 0$. Throughout the paper, to simplify the exposition, we thus assume that
\begin{equation*}
V\geq 0.
\end{equation*}

%  (for instance, denoting
%\begin{equation*}
%V_+:=\max(V,0) , \quad V_-:=\max(-V,0),
%\end{equation*}
%the positive and negative parts of $V$, the KLMN Theorem (see e.g. \cite[Theorem X.17]{ReedSimon2}) shows that if $V_+ \in L^1_{\mathrm{loc}}(\mathbb{R}^d)$ and $V_-$ is relatively form bounded with respect to $-\Delta$ with a relative bound $<1$, then $H$ identifies with a self-adjoint operator). Letting $N_{<0}(H)$ be the number of negative eigenvalues of $H$, one seeks a bound on $N_{<0}(H)$ of the form $N_{<0}(H)\le C(V)$ for some positive constant $C(V)$ to be determined. 

Among the various bounds obtained in the literature, we mention the following ones. The celebrated Cwikel-Lieb-Rozenblum (CLR) bounds state that
\begin{equation}\label{eq:CLR_bound}
N_{<0}(H)\lesssim_{\,d} \int_{\mathbb{R}^d} V^{\frac{d}{2}} , \quad d\ge3,
\end{equation}
for any $V$ in $L^{\frac{d}{2}}$. Throughout this paper, $a\lesssim_{y_1,\dots,y_n} b$ means that there exists a constant $C_{y_1,\dots,y_n}>0$ depending only on the parameters $y_1,\dots,y_n$ such that $a\leq C_{y_1,\dots,y_n}b$, and this constant may change from one line to the other.

The estimates \eqref{eq:CLR_bound} enjoy the important property that they are consistent with Weyl's semi-classical asymptotics. Namely, for sufficiently regular and fast-decaying $V$, 
\begin{equation*}
\lambda^{-\frac{d}{2}}N_{<0}(-\Delta-\lambda V)\to L_d\int_{\mathbb{R}^d} V^{\frac{d}{2}}, \quad \lambda\to\infty,
\end{equation*}
for some positive constant $L_d$ (see e.g. \cite[Section 4.1.1]{lieb_seiringer} or \cite[Theorem 4.28]{FrankLaptevWeidl23}). The CLR bounds were proven independently by Cwikel \cite{Cwikel77}, Lieb \cite{Lieb80} and Rozenblum \cite{Rozenblum76}. They are the crucial endpoint case of a more general family of bounds on the moments of the negative eigenvalues of $H$, the Lieb-Thirring inequalities \cite{LiebThirring76}, that in turn have important consequences for the stability of matter \cite{LiebThirring76_2,lieb_seiringer}. Estimating the best constant in the CLR bound \eqref{eq:CLR_bound} therefore remains a well-studied open problem. We refer to \cite{Hundertmark&al2023} for important recent progress regarding this question and to \cite{Hundertmark&al2023,FrankLaptevWeidl23,Frank21,Schimmer22} for detailed discussions concerning the history, applications, recent developments and open problems related to the Lieb-Thirring inequalities.

Note that the CLR bound \eqref{eq:CLR_bound} implies in particular that if $\|V\|_{L^{d/2}}$ is small enough, in dimension $\geq 3$, then $H$ has no bound states. The situation is different in dimension one or two. In these cases, it is well-known that $H$ has at least one bound state for any $V$ in $\mathrm{C}_0^\infty$ which is not identically zero (see e.g. \cite[Theorem XIII.11]{ReedSimon4}, see also the recent work \cite{Hoangetal23} for similar results for Schr\"odinger operators with general kinetic energies). In one-dimension, the estimate
\begin{equation}\label{eq:estim_dim1}
N_{<0}(H)-1\le  \int_{\mathbb{R}} |x|V(x)\mathrm{d}x , \quad d=1,
\end{equation}
was obtained in \cite{BLANKENBECLER77,KLAUS77}, as a consequence of Bargmann's bound \cite{Bargmann52}. See also \cite[Theorem XIII.9]{ReedSimon4} for other related bounds for central potentials in $3$-dimension.

The two-dimensional case is the most subtle one. In this case it is known that no estimate of the form
\begin{equation*}
N_{<0}(H)\lesssim 1 + \int_{\mathbb{R}} w(x) V(x)\mathrm{d}x ,
\end{equation*}
can hold, provided that $w$ is bounded in a neighborhood of at least one point \cite{Grigor'yan15}. Several papers have been devoted to estimating the number of bound states of $2$-dimensional Schr\"odinger operators in the recent years \cite{Solomyak94,BirmanLaptev96,Chadanetal03,LaptevSolomyak12,LaptevSolomyak13,Shargorodsky14,Grigor'yan15}. In particular, conditions on $V$ ensuring the semi-classical growth $N_{<0}(-\Delta-\lambda V)=\mathcal{O}(\lambda)$ as $\lambda\to\infty$ are derived in \cite{LaptevSolomyak12,LaptevSolomyak13}.
Among the various bounds obtained in $2$-dimension, we mention the following ones:
\begin{equation}\label{eq:estim_dim2}
N_{<0}(H) - 1 \lesssim   \int_{\mathbb{R}^2} (1+\ln\langle x\rangle)V(x)\mathrm{d}x - \int_{|x|\le1} (\ln|x|)V^*(|x|)\mathrm{d}x  , \quad d=2,
\end{equation}
and
\begin{equation}\label{eq:estim_dim2bis}
N_{<0}(H)-1\lesssim \int_{\mathbb{R}^2} (1+\ln\langle x\rangle)V(x)\mathrm{d}x + \|V\|_{L\,\mathrm{log} \, L} , \quad d=2.
\end{equation}
In \eqref{eq:estim_dim2}, $V^*$ stands for the decreasing rearrangement of $V$ defined, for all $t\in[0,\infty)$, by 
\begin{equation*}
V^*(t):=\inf \{ s \in [0,\infty) \, \mid \, \mu_V(s) \le t \},
\end{equation*}
where $\mu_V(s):=|\{x\in\mathbb{R}^2\,|\,|V(x)|>s\}|$. In \eqref{eq:estim_dim2bis}, $\|\cdot\|_{L\,\mathrm{log} \, L}$ stands for the norm in the Orlicz space $L\,\mathrm{log} \, L$ defined by 
\begin{equation*}
\|f\|_{L\,\mathrm{log} \, L}:=\inf\Big\{\kappa>0\, \mid \,\int_{\mathbb{R}^2}\Phi(|f|/\kappa)\le1\Big\},
\end{equation*}
with $\Phi(s)=s\ln(2+s)$ for all $s\in[0,\infty)$. Estimates \eqref{eq:estim_dim2} and \eqref{eq:estim_dim2bis} are proven in \cite{Shargorodsky14}; previously, estimate \eqref{eq:estim_dim2} was proven in the case where $V$ is radial, and conjectured in the general case, in \cite{Chadanetal03}; estimate \eqref{eq:estim_dim2bis} relies on previous important results obtained in \cite{Solomyak94}. We refer to \cite{Shargorodsky14} for further (and stronger) inequalities obtained in the two-dimensional case.

For $0<s<\frac{d}{2}$, one can similarly study the fractional Schr\"odinger operator
\begin{equation}\label{eq:def_Hs}
H_s := (-\Delta)^s - V(x)
\end{equation}
on $L^2(\mathbb{R}^d)$. The proof of the CLR bounds \eqref{eq:CLR_bound} extends to this case, leading to
\begin{equation}\label{eq:CLR_bounds}
N_{<0}(H_s)\lesssim_{\,d,s} \int_{\mathbb{R}^d} V^{\frac{d}{2s}} , \quad d\ge1 ,\quad 0<s<\frac{d}{2}.
\end{equation}
We refer to the review \cite{Frank14} and references therein for bounds on the number of negative eigenvalues and Lieb-Thirring inequalities for $H_s$ with $s<\frac{d}{2}$.

In this paper we consider the fractional Schr\"odinger operator \eqref{eq:def_Hs} in the case $s\geq\frac{d}{2}$. This includes in particular the critical case $s=\frac{d}{2}$, as well as ``polyharmonic Schr\"odinger operators'', namely the fractional Schr\"odinger operators $H_s$ with integer exponent $s\in\mathbb{N}$. For polyharmonic Schr\"odinger operators with $\N\ni s\ge \frac{d}{2}$, it was proven in \cite{EgorovKondratiev91,EgorovKondratiev96} that
\begin{equation}\label{eq:EgKon1}
N_{<0}(H_s)-s\lesssim_{\, s,q} \int_{\mathbb{R}} |x|^{2sq-1}V(x)^q \mathrm{d}x , \quad d=1 , \quad s\in\mathbb{N},\quad q\ge1,
\end{equation}
in one-dimension, and
\begin{align}\label{eq:EgKon2}
&N_{<0}(H_s)-\binom{d+n}{d} \notag \\
&\qquad
\lesssim_{\, d,s,q}\begin{cases}
\displaystyle{\int_{\mathbb{R}^d} |x|^{2sq-d}V(x)^q \mathrm{d}x} , \quad d \text{ odd},\quad s\in\mathbb{N},  \quad q>1, \vspace{0,2cm}\\
\displaystyle{\int_{\mathbb{R}^d} (1+|\ln|x||)^{2q-1}|x|^{2sq-d}V(x)^q} \mathrm{d}x ,\quad d \text{ even},  \quad s\in\mathbb{N}, \quad q>1,
\end{cases}
\end{align}
in any dimension, where $n= \lfloor s- \frac{d}{2} \, \rfloor$. Still for $s\ge d/2$, $s$ an integer, the Lieb-Thirring inequalities for moments of the negative eigenvalues of $H_s$ of order $\mu>1-\frac{d}{2s}$ have been obtained in \cite{Netrusov96}; moreover the asymptotics of $N_{<0}((-\Delta)^s-\lambda V)$ as $\lambda\to\infty$ has been studied in \cite{BirmanLaptev96,BirmanLaptevSolomyak97}, giving in particular sufficient conditions on $V$, for $d$ odd, to ensure the usual semi-classical behavior at large coupling.

Here we aim at proving a bound on $N_{<0}(H_s)$ in any dimension $d$ and for any real $s\ge \frac{d}{2}$, comparable to the bounds of the form \eqref{eq:estim_dim1} or \eqref{eq:EgKon1} (with $q=1$) in dimension one, or \eqref{eq:estim_dim2}--\eqref{eq:estim_dim2bis} in dimension two.

\subsection{Statement of the main result}

Before stating our main results, Theorems \ref{th: N_<0} and \ref{th: N_<0critical},  we recall and introduce some notations.
%For $1\le p,q<\infty$, the spaces $\ell ^q(L^p)$ are defined as follows. For any $\mathbf{m}\in\mathbb{Z}^d$, let $\chi_\mathbf{m}$ be the characteristic function of the unit hypercube of $\R^d$ with center $\mathbf{m}$ and, for all function $f:\mathbb{R}^d\to\mathbb{C}$, let $f_\mathbf{m}:=\chi_\mathbf{m}f$. The space $\ell ^q(L^p)$ is the set of measurable functions $f:\mathbb{R}^d\to\mathbb{C}$ such that $(\|f_\mathbf{m}\|_{L^p})_\mathbf{m} \in \ell^q$, equipped with the norm
%	\begin{equation}
%		\label{def: l^q L^p}
%\|f\|_{\ell ^q(L^p)}:=\left(		\sum_{\mathbf{m}\in\mathbb{Z}^d} \|f_\mathbf{m}\|_{L^p}^q \right)^{1/q} <\infty.
%	\end{equation}
%

The symbol $\mathbb{N}$ denotes the set of integers larger than or equal to 1, and $\mathbb{N}_0:=\mathbb{N}\cup\{0\}$. We use the japanese bracket notation $\langle x\rangle := \sqrt{1+|x|^2}$ for $x\in\mathbb{R}^d$. We recall that for $1\leq p< \infty$, the Schatten ideals $\mathcal{L}^p$ (or trace ideals) and the weak trace ideals $\mathcal{L}^{p,\infty}$ are defined, respectively, as the spaces of compact operators $A$ such that the following quantities are finite:
\begin{equation}
\|A\|_{\mathcal{L}^{p}}:=\Big(\sum_{j\geq0}\lambda_j(A^*A)^{p/2}\Big)^{1/p}, \quad \|A\|_{\mathcal{L}^{p,\infty}}^*:=\sup_{j\geq0}(j+1)^{1/p}\sqrt{\lambda_j(A^*A)},
\end{equation} where $\lambda_j(A^*A)$ is the sequence of the eigenvalues of $A^*A$ sorted in decreasing order.
The star in the notation $\|\cdot\|_{\mathcal{L}^{p,\infty}}^*$ is a reminder that it is a quasinorm but not necessarily a norm. (See e.g. \cite{Simon76} for more information on the weak trace ideals $\mathcal{L}^{p,\infty}$.) Similarly, the space of bounded operators on $L^2$ is denoted by $\mathcal{L}^\infty$.

Our main results are the following.

\begin{theorem}[``Super-critical case'', $s>\frac{d}{2}$]
		\label{th: N_<0}
	Let $d\geq1$, $s>\frac{d}{2}$, $n= \lfloor s- \frac{d}{2}\, \rfloor$ and set~$v:=V^{\frac12}$. Then
\begin{equation*}
	N_{<0}(H_s) -\binom{d+n}{d} \lesssim_{\,d,s}
	\begin{cases}
%			C_{d,p} \big\|\sqrt{1+ \ln\langle x \rangle} \, v\big\|_{L^2} \big\|\sqrt{1+ \ln\langle x \rangle} \, v\big\|_{\ell^2(L^p)} & \textnormal{if }s= \frac{d}{2}\\
			 \big\||x|^{s-\frac{d}{2}} \, v \big\|_{L^2}^2 & \textnormal{if } s-\frac{d}{2} \notin \N_0 ,\vspace{0,1cm}\\
			 \big\|\langle x \rangle ^{s-\frac{d}{2}}\sqrt{1+ \ln\langle x \rangle} \, v\big\|_{L^2}^2 & \textnormal{if } s-\frac{d}{2} \in \N ,
	\end{cases}
\end{equation*}
for all $v$ such that the right hand side is finite. 
\end{theorem}

We have the following accompanying remarks. As usual, for $\alpha\in\mathbb{N}_0^d$ and $x\in\mathbb{R}^d$, we use the notations $x^\alpha=\prod_{j=1}^d x_j^{\alpha_j}$ and $|\alpha|=\sum_{j=1}^d \alpha_j$.

\begin{remark}
The constant $\binom{d+n}{d}$ can be replaced by the possibly smaller constant
\begin{equation*}
\mathrm{c}_{d,n}(v) := \mathrm{dim}\,\,\mathrm{span}\big\{ x^\alpha v \, | \, \alpha\in\mathbb{N}_0^d , \, |\alpha|\le n \big\}
\end{equation*}
(it is not difficult to see that the maximal dimension of the vector space $\mathrm{span}\big\{ x^\alpha v \, | \, \alpha\in\mathbb{N}_0^d , \, |\alpha|\le n \big\}$ is $\binom{d+n}{d}$, see the proof of Theorem \ref{th: N_<0} below). On the other hand, the constant $\mathrm{c}_{d,n}(v)$ cannot be removed from the estimate of Theorem \ref{th: N_<0}, in the sense that there are potentials $V$ in $\mathrm{C}_0^\infty(\mathbb{R}^d)$ such that $H_s$ has at least $\mathrm{c}_{d,n}(v)$ bound states. More precisely, we will prove that for all $V\in\mathrm{C}_0^\infty(\mathbb{R}^d)$, $V\ge0$, the operator $H_s$ has at least $\mathrm{c}_{d,n}(v)$
negative eigenvalues counting multiplicities. See Proposition \ref{prop:binom_constant} below.
\end{remark}

\begin{remark}
In the endpoint case $s=d/2$, the bound stated in Theorem \ref{th: N_<0} does not hold. Indeed, if it were true, then it would imply that, for $d=2$, 
\begin{equation*}
N_{<0}(-\Delta-V(x)) - 1 \lesssim \|\sqrt{1+\ln \langle x\rangle} \, v\|^2_{L^2},
\end{equation*}
 which cannot hold, as discussed in the introduction and proven in \cite{Grigor'yan15}.
\end{remark}

\begin{remark}
Since the operators $H_s$ and $(-\Delta)^s-V(x+x_0)$ are unitarily equivalent, for any $x_0\in\R^d$, the weights $|x|^{s-\frac{d}{2}}$ and $\langle x \rangle ^{s-\frac{d}{2}}\sqrt{1+ \ln\langle x \rangle}$ in the estimate of Theorem \ref{th: N_<0} can be replaced by $|x-x_0|^{s-\frac{d}{2}}$ and $\langle x-x_0 \rangle^{s-\frac{d}{2}}\sqrt{1+ \ln\langle x-x_0 \rangle}$, respectively.
\end{remark}

%\begin{remark}
%	\label{rk: embedding weight sp}
%We recall that, for $1\le q\le p\le\infty$ and $r>d(\frac{1}{q}-\frac{1}{p})$,
%\begin{equation*}
%L^p(\langle x\rangle^{rp}\mathrm{d}x) \hookrightarrow \ell^q(L^p) \hookrightarrow L^p\cap L^q.
%\end{equation*}
%Hence, since $\sqrt{1+\ln \langle x\rangle}\lesssim\langle x\rangle^\varepsilon$ for any $\varepsilon>0$, the estimate in Theorem~\ref{th: N_<0} for $s=\frac{d}{2}$ implies the weaker bound
%\begin{equation*}
%	N_{<0}(H_{\frac{d}{2}})	\le C_{d,p} \| \langle x\rangle^{r} v\|^2_{L^p},
%\end{equation*}
%for any $p>2$ and $r>d(\frac12-\frac{1}{p})$.
%\end{remark}
%
%\begin{remark}
%Again in the critical case $s=\frac{d}{2}$, our proof of Theorem \ref{th: N_<0} actually implies the stronger bound
%\begin{equation*}
%	N_{<0}(H_{\frac{d}{2}})	\le 1+C_{d,p} \sum_{\mathbf{m}\in\mathbb{Z}^{d}} (1+\ln \langle \mathbf{m} \rangle)\|v_\mathbf{m}\|_{L^2}\|v_\mathbf{m}\|_{L^p}.
%\end{equation*}
%\end{remark}

We also note that, for $d=1$ and $s=1$, Theorem \ref{th: N_<0} gives, for the usual Schr\"odinger operator $H=-\Delta-V(x)$,
\begin{equation*}
N_{<0}(H)-1\lesssim \int_{\mathbb{R}}|x| V(x)\mathrm{d}x, \quad d=1.
\end{equation*}
The Bargmann estimate \eqref{eq:estim_dim1}, which follows from the explicit expression of Green's operator in one-dimension, is of course stronger, since it gives the same estimate but with a constant equal to $1$ in front of the integral in the right hand side (instead of the implicit constant we obtain). Likewise, for $d=1$ and $s\in\N$, Theorem \ref{th: N_<0} gives
\begin{equation*}%\label{eq:EgKon1-comp}
N_{<0}(H_s)-s\lesssim_{\, s} \int_{\mathbb{R}} |x|^{2s-1}V(x) \mathrm{d}x , \quad d=1 , \quad s\in\mathbb{N},
\end{equation*}
which is \eqref{eq:EgKon1} in the particular case where $q=1$. Our result therefore shows how \eqref{eq:estim_dim1}, and \eqref{eq:EgKon1} with $q=1$, can be generalized to any dimension for the fractional Schr\"odinger operator $H_s$, with any real $s>d/2$.

For $d$ odd and $\N \ni s\ge d/2$, our result also corresponds to the endpoint case $q=1$ in the family of estimates \eqref{eq:EgKon2} proven by Egorov and Kondratiev \cite{EgorovKondratiev96}. Note that the endpoint case $q=1$ was left open in \cite{EgorovKondratiev96}. Note also that our proof is very different from that in \cite{EgorovKondratiev96}, see the next subsection for a description of the strategy followed in this paper. In the case where $d$ is even, and with $s>d/2$, our result corresponds again to $q=1$ in \eqref{eq:EgKon2}, except for the local behavior of $V$, in that our bound requires that $V$ is $L^1$ near the origin, while \eqref{eq:EgKon2} with $q=1$ would only require that $(\ln|x|) |x|^{2s-d}V(x)$ is $L^1$ near $0$. 

Our result in the critical case $s=d/2$ is stated in terms of the harmonic oscillator
\begin{equation*}
\mathbf{h}:=c_d(-\Delta+x^2),
\end{equation*}
where the constant $c_d$ is chosen, for technical convenience, as $c_d:=e^e/d$ (so that $\mathbf{h}\ge e^e$). We then have the following result.
\begin{theorem}[``Critical case'', $s=\frac{d}{2}$]
		\label{th: N_<0critical}
	Let $d\geq1$, $s=\frac{d}{2}$, $\varepsilon>0$ and set~$v:=V^{\frac12}$. Then 
\begin{equation*}
	N_{<0}(H_s) - 1 \lesssim_{\,d,\varepsilon} \big \|(\ln \mathbf{h})^{\frac12}(\ln\ln \mathbf{h})^{\frac12+\varepsilon}\,v\big\|^2_{L^2} ,
\end{equation*}
for all $v$ such that the right hand side is finite. 
\end{theorem}

Theorem \ref{th: N_<0critical} should be compared with the bounds \eqref{eq:estim_dim2}--\eqref{eq:estim_dim2bis} for $H=-\Delta-V(x)$ in dimension $2$. In particular, similarly as in \eqref{eq:estim_dim2}--\eqref{eq:estim_dim2bis}, our estimate requires both a logarithmic decay and a ``logarithmic regularity'' of $v$, encoded here in the condition that $v$ belongs to the domain of $(\ln \mathbf{h})^{1/2}$. The slightly stronger requirement that $v$ belongs to the (smaller) domain of $(\ln \mathbf{h})^{1/2}(\ln\ln \mathbf{h})^{1/2+\varepsilon}$ may be an artifact of our proof.

\subsection{Elements of the proof and auxiliary results}

Our proof of Theorems~\ref{th: N_<0} and~\ref{th: N_<0critical} starts with a usual application of the Birman-Schwinger principle \cite{Birman61,Schwinger61}. In our context, it states that, for all $E<0$,
\begin{equation}
	\label{comp: B-S}
N_{\le E}(H_s) = N_{\ge 1}(K_E),
\end{equation}
where $N_{\le E}(A)$ (respectively $N_{\ge E}(A)$) denote the number of eigenvalues less than or equal to $E$ (respectively larger than or equal to $E$) of a self-adjoint operator $A$, and the Birman-Schwinger operator $K_E$ is defined by
\begin{equation*}
K_E := v(x) \big( (-\Delta)^s - E \big )^{-1} v(x)\,, \quad E<0\,.
\end{equation*}
Recall that we have set
\begin{equation*}
v:=V^{\frac12}.
\end{equation*}
For the convenience of the reader, a proof of the Birman-Schwinger principle \eqref{comp: B-S} under our assumptions is recalled in Appendix \ref{app: Birman-S}.

Next, recalling that $n= \lfloor s- \frac{d}{2}\, \rfloor$, we introduce the finite-dimensional vector space
\begin{equation}\label{eq:def_Fn}
	\mathcal{F}_n:=\mathrm{span}\big\{ x^\alpha v \, | \, \alpha\in\mathbb{N}_0^d , \, |\alpha|\le n \big\} \subset L^2 \,.
\end{equation}
%Under the assumptions of Theorem~\ref{th: N_<0}, all the $x^\alpha v$ for $|\alpha|\leq n$ are in $L^2$.
The Birman-Schwinger operator is then split into its `low- and high-frequencies' parts. More precisely, we set
\begin{align}
	\label{def: K_0 + K_inf}
 K_{E,<1}:= &  \, v(x) \, \big( (-\Delta)^s - E \big )^{-1} \mathds{1}_{|-i\nabla|<1} \, v(x) \,, & K^\perp_{E,<1} & := \Pi_{\mathcal{F}_n}^\perp K_{E,<1}\Pi_{\mathcal{F}_n}^\perp, \\
 K_{E,>1}:= & \, v(x) \, \big( (-\Delta)^s - E \big )^{-1} \mathds{1}_{|-i\nabla|>1} \, v(x) \,, & K^\perp_{E,>1} & := \Pi_{\mathcal{F}_n}^\perp K_{E,>1}\Pi_{\mathcal{F}_n}^\perp, 	\label{def: K_0 + K_inf2}
\end{align}
where $\Pi_{\mathcal{F}_n}^\perp$ denotes the orthogonal projection onto $\mathcal{F}_n^\perp$.

The variational principle (which we recall in Appendix \ref{app:variational}) then yields
\begin{equation}
	\label{comp: var prin}
N_{\ge 1}(K_E) \le \mathrm{dim}(\mathcal{F}_n) + N_{\ge1}(K_E^\perp),
\end{equation}
where $K_E^\perp = K^\perp_{E,<1} + K^\perp_{E,>1}$. It is not difficult to verify that
\begin{equation*}
\mathrm{dim}(\mathcal{F}_n) \le \binom{d+n}{d},
\end{equation*}
(see Eq. \eqref{eq:dimFn} in the proof of Theorem~\ref{th: N_<0}). Now the splitting into high- and low-frequencies comes into play, as we can write
\begin{equation}
N_{\ge1}(K_E^\perp) \leq 2 \big\|K_{E,>1}^\perp\big\|_{\mathcal{L}^{1,\infty}}^* + 2\big\|K^\perp_{E,<1}\big\|_{\mathcal{L}^{1,\infty}}^* \leq 2 \big\|K_{E,>1}\big\|_{\mathcal{L}^{1,\infty}}^* + 2\big\|K^\perp_{E,<1}\big\|_{\mathcal{L}^1} .
\end{equation}
Note that we have estimated $\|K^\perp_{E,>1}\|_{\mathcal{L}^{1,\infty}}^*\le\|K_{E,>1}\|_{\mathcal{L}^{1,\infty}}^*$, namely we do not use the orthogonal projection $\Pi_{\mathcal{F}_n}^\perp$ for the high-frequencies part. On the other hand, to estimate the low-frequencies part, the orthogonal projection $\Pi_{\mathcal{F}_n}^\perp$ plays a crucial role, but it suffices to estimate the trace norm of $K^\perp_{E,<1}$ instead of the more complicated quasi-norm in $\mathcal{L}^{1,\infty}$.

Theorems~\ref{th: N_<0} and \ref{th: N_<0critical} are then consequences of the following two theorems.
\begin{theorem}[Low-frequencies estimate]\label{thm:low}
Let $d\geq1$, $s\geq d/2$ and $E\le0$. Then
	\begin{equation}\label{eq:low-freq}
		%\big\|\Pi_v^\perp v(x) \big( (-\Delta)^s - E \big )^{-1} \mathds{1}_{|-i\nabla|<1} v (x)\Pi_v^\perp\big\|_{\mathcal{L}^1} 
		\big\|K^\perp_{E,<1}\big\|_{\mathcal{L}^1} 
		\lesssim_{\,d,s}\begin{cases}
%			 C_d \big\| \sqrt{1+\ln\langle x\rangle} \, v\big\|_{L^2}^2 & \textnormal{if } s=\frac{d}{2},\\
			 \big\|\langle x \rangle ^{s-\frac{d}{2}}\sqrt{1+ \ln\langle x \rangle} \, v \big\|_{L^2}^2 & \textnormal{if } s-\frac{d}{2} \in \N_0 , \vspace{0,1cm} \\
			 \big\|\langle x \rangle ^{s-\frac{d}{2}}\, v \big\|_{L^2}^2 & \textnormal{if } s-\frac{d}{2} \notin \N_0 ,
		\end{cases}
	\end{equation}
	for all $v$ such that the right hand side is finite. 
\end{theorem} 
\begin{theorem}[High-frequencies estimate]\label{thm:high}
Let $d\geq1$, $s\geq d/2$, $\varepsilon>0$ 
and $E\leq 0$. Then
\begin{equation}\label{eq:high-frequency-bound}
	\big\|K_{E,>1}\big\|_{\mathcal{L}^{1,\infty}}^*
	\begin{cases}
	\lesssim_{\,d,\varepsilon} \,\big \|(\ln \mathbf{h})^{\frac12}(\ln\ln \mathbf{h})^{\frac12+\varepsilon}\,v\big\|^2_{L^2} 	 & \textnormal{if } s=\frac{d}{2} , \vspace{0,1cm}\\
	\lesssim_{\,d,s} \, \|v\|_{L^2}^2 & \textnormal{if } s>\frac{d}{2} ,
	\end{cases}
\end{equation}
for all $v$ such that the right hand side is finite. 
\end{theorem}

%\begin{equation*}
%	\big\|\Pi_v^\perp v(x) \big( (-\Delta)^s - E \big )^{-1} \mathds{1}_{|-i\nabla|<1} v(x) \Pi_v^\perp\big\|_{\mathcal{L}^{1,\infty}}^* \le \begin{cases}
%		C_d \big\| \sqrt{1+\ln\langle x\rangle} \, v\big\|_{L^2}^2 & \textnormal{if } s=\frac{d}{2},\\
%		C_{d,s} \|\langle x \rangle ^{s-\frac{d}{2}} v \|_{L^2}^2 & \textnormal{if } s-\frac{d}{2} \notin \N_0,\\
%		C_{d,s} \|\langle x \rangle ^{s-\frac{d}{2}}\sqrt{1+ \ln\langle x \rangle} v \|_{L^2}^2 & \textnormal{if } s-\frac{d}{2} \in \N.
%	\end{cases}
%\end{equation*}
%Moreover, since $\Pi_v^\perp$ is a projection, Theorem \ref{thm:high} yields the bound on $\big\|K_{E,>1} \big\|_{\mathcal{L}^{1,\infty}}^*$.
%\begin{align*}
%	&\big\| \Pi_v^\perp v(x) \big( (-\Delta)^s - E \big )^{-1} \mathds{1}_{|-i\nabla|>1} v(x) \Pi_v^\perp \big\|_{\mathcal{L}^{1,\infty}}^* \\
%	&\le\big\| v(x) \big( (-\Delta)^s - E \big )^{-1} \mathds{1}_{|-i\nabla|>1} v(x) \big\|_{\mathcal{L}^{1,\infty}}^*  
%	\le \begin{cases}
%		C_p\|\sqrt{1+\ln \langle x\rangle} \,v\|^2_{\ell^{2}(L^p)}	 & \textnormal{if } s=\frac{d}{2},\\
%		C_s \|v\|_{L^2}^2 & \textnormal{if } s>\frac{d}{2}.
%	\end{cases}
%\end{align*}}
%Hence Theorems~\ref{thm:low} and \ref{thm:high} imply Theorem \ref{th: N_<0}.

The main ideas of the proof of Theorem \ref{thm:low} are as follows. We first use that
\begin{equation*}
\big\|K^\perp_{E,<1}\big\|_{\mathcal{L}^1}=\int_{|\xi|<1} \big\|\Pi_{\mathcal{F}_{n}}^{\perp}e^{ix\cdot\xi}v(x)\big\|_{L_{x}^{2}}^{2}\frac{\mathrm{d}\xi}{|\xi|^{2s}-E} \le \int_{|\xi|<1} \big\|\Pi_{\mathcal{F}_{n}}^{\perp}e^{ix\cdot\xi}v(x)\big\|_{L_{x}^{2}}^{2}\frac{\mathrm{d}\xi}{|\xi|^{2s}} ,
\end{equation*}
(see Lemma \ref{lem:integral-form-for-trace}). For $s\ge\frac{d}{2}$, $\xi\mapsto |\xi|^{-2s}\mathds{1}_{|\xi|<1}$ is not integrable. We decompose the region $|\xi|<1$ into annuli $e^{-k-1}\le|\xi|< e^{-k}$ for $k\in\mathbb{N}_0$, which we combine with a splitting of $v$ in each annuli, of the form $v=v_k^<+v_k^>$, with $v_k^<(x)=\mathds{1}_{|x|\le e^{k}}v(x)$, $v_k^>(x)=\mathds{1}_{|x|\ge e^{k}}v(x)$. For the terms with $v_k^>$, we can use the decay of $v$ at infinity to `gain' powers of $\xi$ since
\begin{equation*}
\big\||\xi|^{-2s}\mathds{1}_{e^{-k-1}\le|\xi|< e^{-k}} \mathds{1}_{|x|\ge e^{k}}v\big\|_{L^2}\lesssim e^{2ks}\big\|\mathds{1}_{|x|\ge e^{k}}v\big\|_{L^2}\le\big\||x|^{2s}v\big\|_{L^2}.
\end{equation*}
A refined estimate shows that the decay conditions imposed in the right-hand side of \eqref{eq:low-freq} are enough to have summability with respect to $k$. To estimate the terms with $v_k^<$, we use that $\Pi_{\mathcal{F}_{n}}^{\perp}x^\alpha v=0$ for all $|\alpha|\le n$. Expanding the exponential $e^{ix\cdot\xi}$ into a series then allows us again to gain powers of $\xi$ and reach integrability.

In the case where $s>\frac{d}{2}$, the proof of Theorem~\ref{thm:high} is straightforward (using that $\xi\mapsto |\xi|^{-2s}\mathds{1}_{|\xi|>1}$ is integrable). In the critical case where $s=\frac{d}{2}$, Theorem~\ref{thm:high} is a corollary of the following Cwikel-type estimate (Theorem \ref{th: norm S^2,infty}). Before stating it we recall a few notations.

For $1\le p < \infty$, the weak spaces $L^{p,\infty}$ are defined as the sets of all measurable functions $f:\mathbb{R}^d\to\mathbb{C}$ such that the quasinorm
\[\|f\|^*_{L^{p,\infty}}:=\sup_{t>0} \lambda(\{|f|>t\})^{1/p}\,t\]
is finite (here $\lambda$ stands for the Lebesgue measure). 
For $1\le p,q<\infty$, the spaces $\ell ^q(L^p)$ are defined as follows. For any $\mathbf{m}\in\mathbb{Z}^d$, let $\chi_\mathbf{m}$ be the characteristic function of the unit hypercube of $\R^d$ with center $\mathbf{m}$ and, for all function $f:\mathbb{R}^d\to\mathbb{C}$, let $f_\mathbf{m}:=\chi_\mathbf{m}f$. The space $\ell ^q(L^p)$ is the set of measurable functions $f:\mathbb{R}^d\to\mathbb{C}$ such that $(\|f_\mathbf{m}\|_{L^p})_\mathbf{m} \in \ell^q$, equipped with the norm
	\begin{equation}
		\label{def: l^q L^p}
\|f\|_{\ell ^q(L^p)}:=\left(		\sum_{\mathbf{m}\in\mathbb{Z}^d} \|f_\mathbf{m}\|_{L^p}^q \right)^{1/q}  \,.
	\end{equation}
Likewise, $\ell^{p,\infty}(\mathbb{Z}^d)$ are the spaces of families of complex numbers $u=(u_{\mathbf{m}})_{\mathbb{Z}^d}$ such that the quasinorm
\[\|u\|^*_{\ell^{p,\infty}} := \sup_{j\geq0}(j+1)^{1/p}u_j^*\]
is finite, where $(u^*_j)_{j\in\mathbb{N}_0}$ is the sequence of the $|u_{\mathbf{m}}|$ sorted in decreasing order.
The space $\ell^{q,\infty}(L^p)$ is defined analogously to the space $\ell^q (L^p)$ in \eqref{def: l^q L^p}. The Fourier transform on $\mathbb{R}^d$ is denoted by
\begin{align*}
\mathcal{F}(f)(\xi)=\hat{f}(\xi)=(2\pi)^{-\frac{d}{2}}\int_{\mathbb{R}^d}e^{-ix\cdot\xi}f(x)\mathrm{d}x.
\end{align*}
For $f:\mathbb{R}^d\to\mathbb{R}$ a measurable function, $f(-i\nabla)$ denotes the operator defined by $f(-i\nabla)\varphi=\mathcal{F}^{-1}(f\hat{\varphi})$.

%
%
%{\color{red}
%INITIAL VERSION
%
%\begin{theorem}[Cwikel-type estimate in $\mathcal{L}^{2, \infty}$]\label{th: norm S^2,infty}
%Let $d\ge 1$. Let $g:\mathbb{R}^d\to[0,\infty)$ be a measurable function such that the following holds: there exists $\delta>0$ such that, for all $p\in(2,2+\delta]$, there are $g_p\in L^{p,\infty}$, $g_{p'}\in\ell^{p',\infty}(L^2)$, with $\frac{1}{p'}=1-\frac{1}{p}$, satisfying $g=g_pg_{p'}$ and
%\begin{equation*}
%\|g_p\|_{L^{p,\infty}}\lesssim1,\quad\|g_{p'}\|_{\ell^{p',\infty}(L^2)}\lesssim1,
%\end{equation*}
%uniformly in $p\in(2,2+\delta]$.
%Then, for all $\varepsilon>0$, there exists $C_{d,\delta,\varepsilon}>0$ such that
%\begin{equation*}%\label{eq:norm_S^2,infty}
%\|f(x) g(-i\nabla)\|_{\mathcal{L}^{2, \infty}}^{*} \leq C _{d,\delta,\varepsilon} \big \| (\ln\mathbf{h})^{\frac12} (\ln\ln\mathbf{h})^{\frac12+\varepsilon} \,f\big\|_{L^2},
%\end{equation*}
%for any $f$ such that the right hand side is finite. 
%%with $g_1 \in L^{p, \infty}, g_2 \in \ell^{p', \infty} (L^2)$ and $\sqrt{1+\ln \langle x\rangle} f \in \ell^{2}(L^p)$.
%\end{theorem}
%}
%
%

\begin{theorem}[Cwikel-type estimate in $\mathcal{L}^{2, \infty}$]\label{th: norm S^2,infty}
Let $d\ge 1$, $\delta>0$ and $\varepsilon>0$. 
Then
\begin{equation*}%\label{eq:norm_S^2,infty}
\|f(x) g(-i\nabla)\|_{\mathcal{L}^{2, \infty}}^{*} 
\lesssim_{\,d,\delta,\varepsilon}
\big \| (\ln\mathbf{h})^{\frac12} (\ln\ln\mathbf{h})^{\frac12+\varepsilon} \,f\big\|_{L^2} \sup_{\substack{2<p\leq 2+\delta \\
\frac{1}{p}+\frac{1}{p'}=1}} \inf_{\substack{g_p,g_{p'}\\ g^2=g_pg_{p'}}} \sqrt{ \|g_p\|_{L^{p,\infty}} \|g_{p'}\|_{\ell^{p',\infty}(L^2)}}
\end{equation*}
for any $f$ and $g$ such that the right hand side is finite. 
%with $g_1 \in L^{p, \infty}, g_2 \in \ell^{p', \infty} (L^2)$ and $\sqrt{1+\ln \langle x\rangle} f \in \ell^{2}(L^p)$.
\end{theorem}

%
%
%{\color{red}
%$g^2$ IN A VECTOR SPACE VERSION:
%
%Given $A$ and $B$ two quasi-normed subspaces of the measurable functions from $\mathbb{R}^d$ to $\mathbb {R}$ endowed with quasi-norms $\|\cdot\|_A$ and  $\|\cdot\|_B$, let
%\[ A\cdot B = \{\sum_{j=1}^J a_j b_j \mid J\in\mathbb{N}, a \in A^J, b \in B^J \} \]
%endowed with the norm
%\[ \|x\|_{A\cdot B} = \inf \Big\{\sum_{j=1}^J \|a_j\|_A \|b_j\|_B \, \mid\, \exists J\in\mathbb{N}, \exists (a,b)\in A^J \times B^J, x=\sum_{j=1}^J a_j b_j \Big\}\,.\]
%\begin{theorem}[Cwikel-type estimate in $\mathcal{L}^{2, \infty}$]\label{th: norm S^2,infty}
%Let $d\ge 1$, $\delta>0$ and $0<\varepsilon<1/2$. 
%Then
%\begin{equation*}%\label{eq:norm_S^2,infty}
%\|f(x) g(-i\nabla)\|_{\mathcal{L}^{2, \infty}}^{*} 
%\lesssim_{d,\delta,\varepsilon}
%\big \| (\ln\mathbf{h})^{\frac12} (\ln\ln\mathbf{h})^{\frac12+\varepsilon} \,f\big\|_{L^2} \sup_{\substack{2<p\leq p+\delta \\
%\frac{1}{p}+\frac{1}{p'}=1
%}}  \|g^2\|_{L^{p,\infty}\, \cdot \, \ell^{p',\infty}(L^2)}^{1/2}
%\end{equation*}
%for any $f$ and $g$ such that the right hand side is finite. 
%%with $g_1 \in L^{p, \infty}, g_2 \in \ell^{p', \infty} (L^2)$ and $\sqrt{1+\ln \langle x\rangle} f \in \ell^{2}(L^p)$.
%\end{theorem}
%}

%
%
\begin{remark}
One can state a slightly stronger estimate, involving the norm of $g$ in a suitably defined vector space, as follows. Given $\mathcal{E}_1$ and $\mathcal{E}_2$ two quasi-normed subspaces of the measurable functions from $\mathbb{R}^d$ to $\mathbb {R}$, endowed with quasi-norms $\|\cdot\|_{\mathcal{E}_1}$ and  $\|\cdot\|_{\mathcal{E}_2}$, consider the vector space
\[\sqrt{\mathcal{E}_1\cdot \mathcal{E}_2} := \Big\{\varphi\mid \exists J\in\mathbb{N},\exists (a,b)\in \mathcal{E}_1^J\times \mathcal{E}_2^J, \varphi^2 \leq \sum_{j=1}^J a_j b_j \Big\}\]
 endowed with the quasi-norm
\[ \|\varphi\|^*_{\sqrt{\mathcal{E}_1\cdot \mathcal{E}_2}} := \inf \Big\{\sqrt{\sum_{j=1}^J \|a_j\|_{\mathcal{E}_1} \|b_j\|_{\mathcal{E}_2}} \, \mid\, J\in\mathbb{N}, (a,b)\in \mathcal{E}_1^J\times \mathcal{E}_2^J, \varphi^2 \leq \sum_{j=1}^J a_j b_j \Big\}\,.\]

Then the following holds: for all $d\ge 1$, $\delta>0$ and $\varepsilon>0$,
\begin{equation*}%\label{eq:norm_S^2,infty}
\|f(x) g(-i\nabla)\|_{\mathcal{L}^{2, \infty}}^{*} 
\lesssim_{\,d,\delta,\varepsilon}
\big \| (\ln\mathbf{h})^{\frac12} (\ln\ln\mathbf{h})^{\frac12+\varepsilon} \,f\big\|_{L^2} \sup_{\substack{2<p\leq 2+\delta \\
\frac{1}{p}+\frac{1}{p'}=1
}}  \|g\|^*_{\sqrt{L^{p,\infty}\, \cdot \, \ell^{p',\infty}(L^2)}}
\end{equation*}
for any $f$ and $g$ such that the right hand side is finite. 
\end{remark}

Theorem \ref{th: norm S^2,infty} is obtained by first decomposing $f$ as
\begin{equation*}
f = \sum_{k\in\mathbb{N}} \pi_k f, \quad \pi_k:=\mathds{1}_{\Lambda_k\le \mathbf{h}<\Lambda_{k+1}}, \quad \Lambda_k:=e^{e^k},
\end{equation*}
and then using H\"older's inequality in weak trace ideals in each spectral region:
\begin{equation*}
\big(\| (\pi_k f)(x)g(-i\nabla)\|_{\mathcal{L}^{2, \infty}}^{*}\big)^2\le \| (\pi_k f)(x)g_p(-i\nabla)\|_{\mathcal{L}^{p, \infty}}^{*} \| (\pi_k f)(x)g_{p'}(-i\nabla)\|_{\mathcal{L}^{p', \infty}}^{*}.
\end{equation*}
Applying the usual Cwikel estimate \cite[Theorem 4.2]{Cwikel77} and an estimate due to Simon \cite[Theorem 4.6]{Simon76}, we are then able to obtain Theorem \ref{th: norm S^2,infty} by  suitably choosing $p$ (depending on~$k$).

\subsection{Organization of the paper} Apart from Cwikel's estimate just mentioned, our paper is self-contained. It is organized as follows. Sections \ref{sec:low} and \ref{sec:high} are devoted to the proofs of Theorems \ref{thm:low} and \ref{thm:high} respectively. In Section \ref{sec:main_thm}, we combine Theorems \ref{thm:low} and \ref{thm:high} to deduce our main results, Theorem~\ref{th: N_<0} and Theorem \ref{th: N_<0critical}. Appendices \ref{app: Birman-S}, \ref{app:variational} and \ref{app:p'} recall proofs of the Birman-Schwinger principle, the variational principle and Simon's result \cite[Theorem 4.6]{Simon76}, respectively.

\bigskip

\noindent \textbf{Acknowledgements.} This research was funded, in whole or in part, by l’Agence Nationale de la Recherche (ANR), project ANR-22-CE92-0013. For the purpose of open access, the author has applied a CC-BY public copyright licence to any Author Accepted Manuscript (AAM) version arising from this submission. We warmly thank R. Frank and M. Lewin for very useful comments. 

%%%%%%%%%%%%%%%%%%%%%%%%%%%%%%%%%%%%%%%%%%%%%%%%%%%%%%%%%
%%%%%%%%%%%%%%%%%%%%%%%%%%%%%%%%%%%%%%%%%%%%%%%%%%%%%%%%%

\section{Low-frequencies estimate}\label{sec:low}

In this section we prove Theorem \ref{thm:low}. We will use the following notations. For $t\ge0$, let $\sigma_t:=e^{-t}$. We decompose $v$ into
\begin{equation}\label{eq:defv_t}
	v_t^<(x):=\mathds{1}_{|x|\le e^{t}}v(x)\quad \text{and} \quad v_t^>(x):=\mathds{1}_{|x|\ge e^{t}}v(x).
\end{equation}

Before we turn to the proof of Theorem~\ref{thm:low}, we prove the following easy lemma which gives a convenient formula for the trace of  $\Pi_{\mathcal{F}_n}^\perp K_{E,<1} \Pi_{\mathcal{F}_n}^\perp$ (recall that $K_{E,<1}$ has been defined in \eqref{def: K_0 + K_inf} and $\mathcal{F}_n$ has been defined in \eqref{eq:def_Fn}). Note that taking $B=\mathrm{Id}$ in the next lemma, we obtain the well-known formula for the Hilbert-Schmidt norm of an operator of the form $g(-i\nabla)f(x)$.

\begin{lemma}\label{lem:integral-form-for-trace}
Let $f$, $g$ be two functions in $L^2$ and $B$ be a bounded operator on $L^2$. Then
\begin{equation}
\big\| \bar{g}(-i\nabla) \bar{f}(x) B^* \big\|_{\mathcal{L}^2}^2
=(2\pi)^{-d}\int_{\mathbb{R}^d} |g(\xi)|^2 \| B e^{ix\cdot\xi} f(x) \|_{L^2_x}^2 \mathrm{d}\xi \,.
\end{equation}
\end{lemma}

\begin{proof}
Let $(\varphi_j)_{j\in\mathbb{N}_0}$ be an orthonormal basis of $L^2$. For all $j\in\mathbb{N}_0$, we have
\begin{align}\label{eq:lemmaB_1}
\big\|\bar{g}(-i\nabla)\bar f(x)B^*\varphi_j\big\|_{L^2}^2=\int_{\mathbb{R}^d}|g(\xi)|^2\big|\mathcal{F}(\bar f(x)B^*\varphi_j)(\xi)\big|^2\mathrm{d}\xi.
\end{align}
Now, for all $\xi\in\mathbb{R}^d$, we can rewrite
\begin{align*}
\mathcal{F}(\bar f(x)B^*\varphi_j)(\xi)=(2\pi)^{-\frac{d}{2}}\big\langle e^{ix\cdot\xi} f(x) , B^*\varphi_j \big\rangle_{L^2_x}=(2\pi)^{-\frac{d}{2}} \big\langle Be^{ix\cdot\xi} f(x) , \varphi_j \big\rangle_{L^2_x}.
\end{align*}
Summing \eqref{eq:lemmaB_1} over $j$, we obtain
\begin{align*}
\big\|\bar{g}(-i\nabla)\bar f(x)B^*\big\|_{\mathcal{L}^2}^2=(2\pi)^{-d} \sum_{j\in\mathbb{N}_0} \int_{\mathbb{R}^d}|g(\xi)|^2\big| \big\langle Be^{ix\cdot\xi} f(x) , \varphi_j \big\rangle_{L^2_x} \big|^2\mathrm{d}\xi,
\end{align*}
which implies the statement of the lemma by Parseval's equality.
\end{proof}

%\begin{proof}
% The Fourier transform of a function $f\in L^1$ is defined by 
%\begin{equation*}
%	\hat{f}(\xi)=(2\pi)^{-\frac{d}{2}}\int_{\mathbb{R}^d}e^{-ix\cdot\xi}f(x)\mathrm{d}x.
%\end{equation*}
%and extended as usual to tempered distributions.
%
%\end{proof}

Now we are ready to prove Theorem~\ref{thm:low}.

\begin{proof}[Proof of Theorem~\ref{thm:low}]
Applying Lemma~\ref{lem:integral-form-for-trace} we can express the trace of $K_{E,<1}^\perp$ as
\[
\big \|\Pi_{\mathcal{F}_{n}}^{\perp}v(x)\big((-\Delta)^{s}-E\big)^{-1}\mathds{1}_{|-i\nabla|<1}v(x)\Pi_{\mathcal{F}_{n}}^{\perp} \big\|_{\mathcal{L}^{1}}=\int_{|\xi|<1} \big\|\Pi_{\mathcal{F}_{n}}^{\perp}e^{ix\cdot\xi}v(x)\big\|_{L_{x}^{2}}^{2}\frac{\mathrm{d}\xi}{|\xi|^{2s}-E}\,.
\]
Using the decompositions $e^{i\theta}=\sum_{j=0}^{n}\frac{(i\theta)^{j}}{j!}+\sum_{j\ge n+1}\frac{(i\theta)^{j}}{j!}$
and $v=v_{k}^{>}+v_{k}^{<}$, we obtain, since $E\le0$,
\[
\int_{|\xi|<1} \big\|\Pi_{\mathcal{F}_{n}}^{\perp}e^{ix\cdot\xi}v(x)\big\|_{L_{x}^{2}}^{2}\frac{\mathrm{d}\xi}{|\xi|^{2s}}\lesssim\sum_{k\in\mathbb{N}_{0}}\int_{\sigma_{k+1}\leq|\xi|<\sigma_{k}}(A_1(k,\xi)+A_2(k,\xi)+B(k,\xi)) \, \frac{\mathrm{d}\xi}{|\xi|^{2s}},
\]
where we have set
\begin{equation*}
A_1(k,\xi)  := \big\|\Pi_{\mathcal{F}_{n}}^{\perp}e^{ix\cdot\xi}v_{k}^{>}(x)\big\|_{L_{x}^{2}}^{2} \,, \quad
A_2(k,\xi)  := \sum_{j=0}^{n}\Big\|\Pi_{\mathcal{F}_{n}}^{\perp} \frac{(ix\cdot\xi)^{j}}{j!}v_{k}^{<}(x)\Big\|_{L_{x}^{2}}^{2},
\end{equation*}
and
\begin{equation*}
B(k,\xi)  := \Big\|\Pi_{\mathcal{F}_{n}}^{\perp}\sum_{j\ge n+1}\frac{(ix\cdot\xi)^{j}}{j!}v_{k}^{<}(x)\Big\|_{L_{x}^{2}}^{2}\,.
\end{equation*}

The estimate of $A_1(k,\xi)$ is straightforward:
\begin{equation*}
A_1(k,\xi) 
\leq \|e^{ix\cdot\xi}v_{k}^{>}(x)\|_{L_{x}^{2}}^{2}
\leq \|v_{k}^{>}\|_{L_{x}^{2}}^{2}
\le\sum_{|\alpha|\leq n} |\xi|^{2|\alpha|} \| x^{\alpha}v_{k}^{>}\|_{L^{2}}^{2} \,.
\end{equation*}
The purpose of the last inequality is only to bound $A_1(k,\xi)$ and $A_2(k,\xi)$ by the same term. To estimate $A_2(k,\xi)$, using $|\xi|\leq1$ and $\Pi_{\mathcal{F}_{n}}^{\perp}x^{\alpha}v=0$ for $|\alpha|\leq n$, we write
\begin{equation*}
A_2(k,\xi) 
 \leq \sum_{|\alpha|\leq n} \xi^{2\alpha} \big\|\Pi_{\mathcal{F}_{n}}^{\perp}x^{\alpha}(v-v_{k}^{>})\big\|_{L^{2}}^{2}
 = \sum_{|\alpha|\leq n} \xi^{2\alpha} \big\|\Pi_{\mathcal{F}_{n}}^{\perp}x^{\alpha}v_{k}^{>}\big\|_{L^{2}}^{2}
 \leq \sum_{|\alpha|\leq n} |\xi|^{2|\alpha|} \| x^{\alpha}v_{k}^{>}\|_{L^{2}}^{2} \,.
\end{equation*}
Integrating over $\xi$ and summing over $k$ yields
\begin{align*}
\sum_{k\in\mathbb{N}_{0}}\int_{\sigma_{k+1}\leq|\xi|<\sigma_{k}}(A_1(k,\xi)+A_2(k,\xi))\,\frac{\mathrm{d}\xi}{|\xi|^{2s}} 
& \lesssim\sum_{\substack{k\in\mathbb{N}_{0}\\ |\alpha|\leq n}} \int_{\sigma_{k+1}\leq|\xi|<\sigma_{k}} \frac{\mathrm{d}\xi}{ |\xi|^{2s-2|\alpha|}} \| x^\alpha v_{k}^{>}\|_{L^{2}}^{2}\\
 & \lesssim_{\,d,s} \sum_{\substack{k\in\mathbb{N}_{0}\\ |\alpha|\leq n}} \sigma_{k}^{d-2s+2|\alpha|}\| x^\alpha v_{k}^{>}\|_{L^{2}}^{2}.
% & \lesssim\begin{cases}
%\|\sqrt{1+\ln\langle x\rangle}v\|_{L^{2}}^{2} & \text{if}\quad s=d/2\\
%\|\langle x\rangle^{s-d/2+n}v\|_{L^{2}}^{2} & \text{if}\quad s>d/2
%\end{cases}
\end{align*}
To bound this sum by an integral we isolate the term for $k=0$, shift the indexes for $k\geq 1$ and use that $\sigma_{k+1} = e^{-1} \sigma_k$ to obtain
\begin{equation*}
\sum_{k\in\mathbb{N}_{0}}\int_{\sigma_{k+1}\leq|\xi|<\sigma_{k}}(A_1(k,\xi)+A_2(k,\xi))\,\frac{\mathrm{d}\xi}{|\xi|^{2s}} \\
%	\sum_{\substack{k\in\mathbb{N}_{0}\\ |\alpha|\leq n}} \sigma_{k}^{d-2s+2|\alpha|}\| x^\alpha v_{k}^{>}\|_{L^{2}}^{2}&
%&	\lesssim \sum_{|\alpha|\leq n}\| x^\alpha v\|_{L^2}^2+\sum_{\substack{k\in\mathbb{N}_{0}\\ |\alpha|\leq n}}\sigma_{k}^{d-2s+2|\alpha|}\| x^\alpha v_{k+1}^>\|_{L^2}^2 \\
\lesssim_{\,d,s} \|\langle x\rangle^n v\|_{L^2}^2+\sum_{\substack{k\in\mathbb{N}_{0}\\ |\alpha|\leq n}} \sigma_{k}^{d-2s+2|\alpha|}\| x^\alpha v_{k+1}^>\|_{L^2}^2 .
\end{equation*}
Now, since $k\mapsto\sigma_k^{d-2s+2|\alpha|}$ is non-decreasing (given that $s-\frac{d}{2}\geq n\geq |\alpha|$) and $k\mapsto \| x^\alpha v_k^>\|_{L^2}^2$ is decreasing, we can estimate
\begin{multline*}
	\sum_{k\in\mathbb{N}_0}\sigma_{k}^{d-2s+2|\alpha|}\| x^\alpha v_{k+1}^>\|_{L^2}^2
	=\sum_{k\in\mathbb{N}_{0}} \int_k^{k+1}\sigma_{k}^{d-2s+2|\alpha|}\| x^\alpha v_{k+1}^>\|_{L^2}^2\mathrm{d}t\\
	\le \sum_{k\in\mathbb{N}_0}\int_k^{k+1}\sigma_{t}^{d-2s+2|\alpha|}\| x^\alpha v_{t}^>\|_{L^2}^2 \mathrm{d}t
%	&= \int_0^\infty\sigma_{t}^{d-2s+2|\alpha|}\| x^\alpha v_{t}^>\|_{L^2}^2 \mathrm{d}t\\
		= \int_0^{\infty} \sigma_{t}^{d-2s+2|\alpha|} \int_{|x|\ge e^t}| x^\alpha v(x)|^2\mathrm{d}x\mathrm{d}t .
\end{multline*}
	By Fubini's Theorem, this gives 
\begin{multline}
\sum_{k\in\mathbb{N}_{0}}\int_{\sigma_{k+1}\leq|\xi|<\sigma_{k}}(A_1(k,\xi) + A_2(k,\xi)) \frac{\mathrm{d}\xi}{|\xi|^{2s}} \\
\lesssim_{\,d,s} \|\langle x\rangle^{n}v\|_{L^2}^2+\sum_{|\alpha|\leq n} \int_{|x|\ge1} |x^\alpha v(x)|^2\int_0^{\ln|x|}\sigma_{t}^{d-2s+2|\alpha|}\,\mathrm{d}t\,\mathrm{d}x   \\% \nonumber  \\ 
	\lesssim_{\,d,s}
	\begin{cases}
		\| \langle x\rangle^{s-\frac{d}{2}}\sqrt{1+\ln\langle x\rangle} v\|_{L^2}^2  & \textnormal{ if } s-\frac{d}{2}\in\mathbb{N}_0,\\
		\| \langle x\rangle^{s-\frac{d}{2}} v\|_{L^2}^2 & \textnormal{ if } s-\frac{d}{2}\notin\mathbb{N}_0,
	\end{cases}
%\label{eq:v_k> low d}
\label{eq:bound-A}
\end{multline}
where we used that $\sigma_t=e^{-t}$ in the last inequality.
%\begin{equation}
%\le C_{d,s}
%	\begin{cases}
%		\| \sqrt{1+\ln\langle x\rangle} v\|_{L^2}^2  & \textnormal{ if } s= \frac{d}{2}\\
%		\| \langle x\rangle^{s-d/2+n} v\|_{L^2}^2 & \textnormal{ if } s >\frac{d}{2}
%	\end{cases}\,.
%\end{equation}

It remains to estimate $B(k,\xi)$. We write
\begin{equation*}
B(k,\xi)
 \leq\Big\|\sum_{j\ge n+1}\frac{(ix\cdot\xi)^{j}}{j!}v_{k}^{<}(x)\Big\|_{L_{x}^{2}}^{2}
  \lesssim_{\,d}\sum_{|\alpha|=n+1}\xi^{2\alpha}\|x^{\alpha}v_{k}^{<}\|_{L^{2}}^{2}
  \lesssim_{\,d} |\xi|^{2n+2}\big\||x|^{n+1}v_{k}^{<}\big\|_{L^{2}}^{2} \,.
\end{equation*}
Integrating over $\xi$ and summing over $k$ yields
\begin{align*}
\sum_{k\in\mathbb{N}_{0}}\int_{\sigma_{k+1}\leq|\xi|<\sigma_{k}}B(k,\xi)\frac{\mathrm{d}\xi}{|\xi|^{2s}} &
\lesssim_{\,d} \sum_{k\in\mathbb{N}_{0}}\int_{\sigma_{k+1}\leq |\xi|<\sigma_{k}}|\xi|^{2n+2-2s}\mathrm{d}\xi\: \big\||x|^{n+1}v_{k}^{<}\big\|_{L^{2}}^{2}\\
% & \lesssim\sum_{k\in\mathbb{N}_{0}}\int_{\sigma_{k+1}}^{\sigma_{k}}\rho^{d-1+2n+2-2s}\mathrm{d}\rho\:\||x|^{n+1}v_{k}^{<}(x)\|_{L_{x}^{2}}^{2}\\
 & \lesssim_{\,d} \sum_{k\in\mathbb{N}_{0}}\sigma_{k}^{d-2s+2n+2}\:\big\||x|^{n+1}v_{k}^{<}\big\|_{L^{2}}^{2} \,.
\end{align*}
Since $k\mapsto\sigma_{k}^{d-2s+2n+2}$ is decreasing (as $d-2s+2n+2>0$)
and $t\mapsto\||x|^{n+1}v_{t}^{<}\|_{L^{2}}^{2}$ is increasing, we can estimate
\begin{align*}
\sum_{k\in\mathbb{N}_{0}}\int_{\sigma_{k+1}\leq|\xi|<\sigma_{k}}B(k,\xi)\frac{\mathrm{d}\xi}{|\xi|^{2s}}  
& \lesssim_{\,d} \sum_{k\in\mathbb{N}_{0}}\int_{k}^{k+1}\sigma_{k+1}^{d-2s+2n+2}\:\big\||x|^{n+1}v_{k}^{<}\big\|_{L^{2}}^{2}\mathrm{d}t\\
 & \lesssim_{\,d}\sum_{k\in\mathbb{N}_{0}}\int_{k}^{k+1}\sigma_{t}^{d-2s+2n+2}\:\big\||x|^{n+1}v_{t}^{<}\big\|_{L^{2}}^{2}\mathrm{d}t\\
 & \lesssim_{\,d}\int_{0}^{\infty}\sigma_{t}^{d-2s+2n+2}\:\int_{|x|<e^{t}}|x|^{2n+2}|v(x)|^{2}\,\mathrm{d}x\,\mathrm{d}t.
\end{align*}
An application of Fubini's theorem yields, as $\sigma_t=e^{-t}$,
\begin{align*}
\sum_{k\in\mathbb{N}_{0}}\int_{\sigma_{k+1}\leq|\xi|<\sigma_{k}}B(k,\xi)\frac{\mathrm{d}\xi}{|\xi|^{2s}} 
 %& \lesssim\int_{\mathbb{R}^{d}}\int_{\ln|x|}^{\infty}\sigma_{t}^{d-2s+2n+2}\mathrm{d}t\,|x|^{2n+2}|v(x)|^{2}\,\mathrm{d}x\\
 & \lesssim_{\,d}\int_{\mathbb{R}^{d}}\int_{\ln|x|}^{\infty}e^{-t(d-2s+2n+2)}\mathrm{d}t\,|x|^{2n+2}|v(x)|^{2}\,\mathrm{d}x\,\\
 & \lesssim_{\,d,s}\int_{\mathbb{R}^{d}}|x|^{-(d-2s+2n+2)}\,|x|^{2n+2}|v(x)|^{2}\,\mathrm{d}x\,\\
 & \lesssim_{\,d,s}\int_{\mathbb{R}^{d}}|x|^{2s-d}|v(x)|^{2}\,\mathrm{d}x\,,
\end{align*}
and therefore
\begin{equation}\label{eq:bound-B}
\sum_{k\in\mathbb{N}_{0}}\int_{\sigma_{k+1}\leq|\xi|<\sigma_{k}}B(k,\xi)\frac{\mathrm{d}\xi}{|\xi|^{2s}}\lesssim_{\,d,s}\big\|\langle x\rangle^{s-\frac{d}{2}}v\big\|_{L^{2}}^{2}\,.
\end{equation}
Putting together \eqref{eq:bound-A} and \eqref{eq:bound-B} we obtain the statement of Theorem~\ref{thm:low}.
\end{proof}

\section{High-frequencies estimate}\label{sec:high}
%This section is devoted to the proof of the Cwikel-type estimate in $\mathcal{L}^{2,\infty}$ given in Theorem~\ref{th: norm S^2,infty}, as well as its consequence stated in Theorem~\ref{thm:high}. Before giving the proof of Theorem~\ref{th: norm S^2,infty}, we show that it indeed implies Theorem~\ref{thm:high}.

This section is devoted to the proof of the Cwikel-type estimate in $\mathcal{L}^{2,\infty}$ given in Theorem~\ref{th: norm S^2,infty}, as well as its consequence stated in Theorem~\ref{thm:high}. Before giving the proof of Theorem~\ref{th: norm S^2,infty}, we show that it indeed implies Theorem~\ref{thm:high}. 

\begin{proof}[Proof of Theorem~\ref{thm:high} using Theorem~\ref{th: norm S^2,infty}]
Recall that $K_{E,>1}$ has been defined in \eqref{def: K_0 + K_inf2}. As $x\mapsto 1/x$ is operator monotone, we have
\begin{equation}
	\label{eq: K_E>1}
\|K_{E,>1}\|_{\mathcal{L}^{1,\infty}}^* 
% \le \|  K_{0,>1}\|_{\mathcal{L}^{1,\infty}}^* 
\le \|v(x)\, (-\Delta)^{-s}  \mathds{1}_{|-i\nabla|>1}\,v(x)\|_{\mathcal{L}^{1,\infty}}^* ,
\end{equation}
for all $E\le0$, the operator $(-\Delta)^{-s}  \mathds{1}_{|-i\nabla|>1}$ being bounded.

For $s>\frac{d}{2}$, the map $\xi\mapsto |\xi|^{-s}\mathds{1}_{|\xi|\ge1}$ belongs to $L^2$. Hence the statement of Theorem \ref{thm:high} is straightforward since the trace norm dominates the $\|\cdot\|_{\mathcal{L}^{1,\infty}}^*$-norm and 
	\begin{align*}
		\big\|v(x) (-\Delta)^{-s} \mathds{1}_{|-i\nabla|>1}v(x)\big\|_{\mathcal{L}^{1}} &=  \big\| (-\Delta)^{-\frac{s}{2}} \mathds{1}_{|-i\nabla|>1}v(x)\big\|^2_{\mathcal{L}^{2}} \\
		&=\big\||\xi|^{-s}\mathds{1}_{|\xi|\ge1}\big\|^2_{L^2}\|v\|^2_{L^2}=C_s\|v\|^2_{L^2}.
\end{align*}

For $s=\frac{d}{2}$,
 writing $v(x) (-\Delta)^{-d/2} \mathds{1}_{|-i\nabla|>1}v(x)= AA^*$ with $A=  v(x) (-\Delta)^{-d/4} \mathds{1}_{|-i\nabla|>1}$, together with the relation $ \|A^* A\|_{\mathcal{L}^{1,\infty}}^*=(\|A\|_{\mathcal{L}^{2,\infty}}^{*})^2$,
yields
	\begin{align}
		\big\|v(x) (-\Delta)^{-d/2} \mathds{1}_{|-i\nabla|>1}v(x)\big\|_{\mathcal{L}^{1,\infty}}^*  = \big(\big\|v(x) (-\Delta)^{-d/4} \mathds{1}_{|-i\nabla|>1} \big\|_{\mathcal{L}^{2,\infty}}^{*}\big)^2.\label{eq:a_infty}
\end{align}
Now we apply Theorem~\ref{th: norm S^2,infty} with $f(x)=v(x)$, $g(\xi) = \frac{1}{|\xi |^{d/2}}\mathds{1}_{|\xi|\ge1}$. Setting 
\begin{equation*}
g_p(\xi):=|\xi|^{-d/p}\mathds{1}_{|\xi|\ge1},
\end{equation*}
we have $g^2=g_pg_{p'}$ for any $p\ge2$ and $\frac{1}{p'}=1-\frac{1}{p}$. We claim that the quasinorms
	\begin{equation*}
		\|g_p\|_{L^{p,\infty}}^*, \quad \|g_{p'}\|_{\ell^{p',\infty}(L^2)}^*
	\end{equation*}
	are uniformly bounded with respect to $p\geq 2$. Indeed, an easy computation shows that 
	\begin{align}\label{eq:comput_gp}
		\|g_p\|_{L^{p,\infty}}^* =& \sup_{t>0} t \lambda(\{1\leq |\xi| \leq t^{-p/d}\})^{1/p}\lesssim_{\,d} 1.
	\end{align}
 Similarly, for $g_{p'}$,
\begin{align}
	\|g_{p'}\|_{\ell^{p',\infty}(L^2)}^* =& \|\, \|\chi_{\mathbf{m}} g_{p'}\|_{L^2}\|_{\ell^{p',\infty}}^* \lesssim \| \langle \mathbf{m}\rangle^{-d/p'} \|_{\ell^{p',\infty}}^* = \sup_{j \geq 0 } (j+1)^{1/p'} (\langle \mathbf{m}\rangle^{-d/p'})_j^*\notag	\\
	\lesssim_{\,d} &  \sup_{j\geq 1}(j+1)^{1/p'} j^{-1/p'} \lesssim_{\,d} 2. \label{eq:comput_gp'}
\end{align}
Hence we can apply Theorem~\ref{th: norm S^2,infty} with
\begin{equation*}
\sup_{\substack{2<p\leq 2+\delta \\
\frac{1}{p}+\frac{1}{p'}=1}} \inf_{\substack{g_p,g_{p'}\\ g^2=g_pg_{p'}}} \sqrt{ \|g_p\|_{L^{p,\infty}} \|g_{p'}\|_{\ell^{p',\infty}(L^2)}} \ \lesssim_{\,d} 1,
\end{equation*}
for any $\delta>0$. This concludes the proof of Theorem~\ref{thm:high}.
\end{proof}

Now we turn to the proof of Theorem~\ref{th: norm S^2,infty}. It is based on the following results. 
\begin{theorem}[Cwikel \cite{Cwikel77}]\label{thm:Cwikel}
Let $d\ge1$. Then
\begin{equation*}
\| f(x)g(-i\nabla) \|_{\mathcal{L}^{p,\infty}}^* \lesssim_{\,d} (p-2)^{-\frac{1}{p}}\|f\|_{L^p} \|g\|_{L^{p,\infty}}^*
\end{equation*}
for all $p\in(2,\infty)$, $f\in L^p(\mathbb{R}^d)$ and $g\in L^{p,\infty}(\mathbb{R}^d)$.
\end{theorem}
\begin{lemma}[Sobolev embedding]\label{lm:sobolev}
Let $d\ge1$, $\delta>0$. Then
\begin{equation*}
\|f\|_{L^p}\lesssim_{\,d,\delta}\|f\|_{H^t}
\end{equation*}
for all $p\in[2,2+\delta]$, $t\ge d(\frac12-\frac1p)$ and $f\in H^t$.
\end{lemma}
\begin{theorem}[Simon \cite{Simon05}]\label{thm:simon}
Let $d\ge1$, $0<\delta'<1$. Then
\begin{equation*}
\| f(x)g(-i\nabla) \|_{\mathcal{L}^{p',\infty}}^* \lesssim_{\,d,\delta'} (2-p')^{\frac{1}{p'}-1}\|f\|_{L^{p'}} \|g\|_{\ell^{p',\infty}(L^2)}^*
\end{equation*}
for all $p'\in[2-\delta',2)$, $f\in L^{p'}(\mathbb{R}^d)$ and $g\in \ell^{p',\infty}(L^2(\mathbb{R}^d))$.
\end{theorem}
\begin{lemma}[Embedding of $L^2(\langle x\rangle^{2r}\mathrm{d}x)$ into $\ell^{p'}(L^2)$]\label{lm:embedding}
Let $d\ge1$. Then
\begin{equation*}
\|f\|_{\ell^{p'}(L^2)}\lesssim_{\,d}\|\langle x\rangle^rf\|_{L^2}.
\end{equation*}
for all $1\le p'<2$, $r>d(\frac{1}{p'}-\frac{1}{2})$ and $f\in L^2(\langle x\rangle^{2r}\mathrm{d}x)$.
\end{lemma}
Theorem \ref{thm:Cwikel} is a direct consequence of \cite[Theorem 4.2]{Cwikel77}, Lemma \ref{lm:sobolev} is the usual Sobolev embedding, Theorem \ref{thm:simon} is \cite[Theorem 4.6]{Simon05} with an explicit dependence on the parameter $p'$, and Lemma \ref{lm:embedding} follows from a direct computation. In Appendix \ref{app:p'}, for the convenience of the reader, we prove Theorem \ref{thm:simon}, reproducing the proof of \cite[Theorem 4.6]{Simon05} and following the dependence on $p'$ in each estimate, and we prove Lemma \ref{lm:embedding}.

 We recall from the introduction the definition of the harmonic oscillator
\begin{equation*}
\mathbf{h}:=c_d(-\Delta+x^2),
\end{equation*}
where the constant $c_d$ is chosen such that $\mathbf{h}\ge e^e$.

\begin{proof}[Proof of Theorem~\ref{th: norm S^2,infty}]	
%Let $g:\mathbb{R}^d\to[0,\infty)$ and $\delta>0$ be as in the statement of the theorem.
Without loss of generality, we assume that $0<\delta<1$
% and $0<\varepsilon<\frac12$ (since $\mathbf{h}\ge e^e$, the norm $\big \| (\ln\mathbf{h})^{\frac12} (\ln\ln\mathbf{h})^{\frac12+\varepsilon} \,f\big\|_{L^2}$ appearing in the statement of Theorem~\ref{th: norm S^2,infty} is increasing in $\varepsilon>0$)
. Let $\Lambda_k:=e^{e^k}$. We will use the following decomposition:
\begin{equation*}
f = \sum_{k\in\mathbb{N}} \pi_k f,
\end{equation*}
where $\pi_k$ stands for the spectral projection
\begin{equation*}
\pi_k:=\mathds{1}_{\Lambda_k\le \mathbf{h}<\Lambda_{k+1}}.
\end{equation*}
Using that $\|\cdot\|_{\mathcal{L}^{2,\infty}}^{*}$ is equivalent to a certain norm $\|\cdot\|_{\mathcal{L}^{2,\infty}}$, we can write
\begin{align}
\big\|f(x) g(-i\nabla)\big\|_{\mathcal{L}^{2,\infty}}^{*}&\lesssim\big\|f(x) g(-i\nabla)\big\|_{\mathcal{L}^{2,\infty}} \notag \\
&\lesssim\sum_{k\in\mathbb{N}}\big\| (\pi_kf)(x) g(-i\nabla)\big\|_{\mathcal{L}^{2,\infty}}\lesssim\sum_{k\in\mathbb{N}}\big\| (\pi_kf)(x) g(-i\nabla)\big\|_{\mathcal{L}^{2,\infty}}^*. \label{eq:double_ln}
\end{align}

%	\begin{align}
%			\big\| v(x)g_1(-i\nabla)v(x)\big\|_{\mathcal{L}^{1,\infty}}^* & =\big\| g_2(-i\nabla)|v(x)|^2 g_2(-i\nabla)\big\|_{\mathcal{L}^{1,\infty}}^*\notag\\
%		& =\Big\|g_2(-i\nabla)\sum_{k\in\mathbb{N}}|\pi_kv(x)|^{2}g_2(-i\nabla)\Big\|_{\mathcal{L}^{1,\infty}}^*\notag \\
%		 & \lesssim\sum_{k\in\mathbb{N}} (1+\ln k) \,\|g_2(-i\nabla)|\pi_kv(x)|^{2} g_2(-i\nabla)\|_{\mathcal{L}^{1,\infty}}^*, \label{eq:double_ln}
%	\end{align}
%where in the last inequality we used the subadditivity of Proposition \ref{Pro:Logarithmic-triangular-inequality-series}.

%Using that $\|\cdot\|_{\mathcal{L}^{2,\infty}}^{*}$ is equivalent to a certain norm $\|\cdot\|_{\mathcal{L}^{2,\infty}}$, we can write
%\begin{align*}
%\|v(x) g_1(-i\nabla)\|_{\mathcal{L}^{2,\infty}}^{*}&\lesssim\|v(x) g_1(-i\nabla)\|_{\mathcal{L}^{2,\infty}}\\
%&\lesssim\sum_{k\in\mathbb{N}_0}\| (\pi_kv)(x) g_1(-i\nabla)\|_{\mathcal{L}^{2,\infty}}\lesssim\sum_{k\in\mathbb{N}_0}\| (\pi_kv)(x) g_1(-i\nabla)\|_{\mathcal{L}^{2,\infty}}^*.
%\end{align*}

Let $p\in(2,2+\delta]$, $\frac{1}{p}+\frac{1}{p'}=1$, and let $g_p\in L^{p,\infty}$, $g_{p'}\in \ell^{p',\infty}(L^2)$ be such that $g^2=g_pg_{p'}$. 
Thanks to the relation $(\|A\|_{\mathcal{L}^{2,\infty}}^{*})^2 = \|A^* A\|_{\mathcal{L}^{1,\infty}}^*$ for any operator $A$ in $\mathcal{L}^{2,\infty}$, we have, for all $k\in\mathbb{N}$,
\begin{align}
\big(\big\| (\pi_kf)(x) \,  g(-i\nabla)\big\|_{\mathcal{L}^{2,\infty}}^*\big)^2 &=\big\|g(-i\nabla) (\pi_kf(x))^{2} g(-i\nabla)\big\|_{\mathcal{L}^{1,\infty}}^*  \nonumber \\
&= \|\pi_kf(x) \, g^2(-i\nabla) \, \pi_kf(x)\|_{\mathcal{L}^{1,\infty} }^* \nonumber \\
	&= \| \pi_kf(x)  \, g_p(-i\nabla) \, \big(\pi_kf(x) g_{p'}(-i\nabla)\big)^* \|_{\mathcal{L}^{1,\infty}}^* \nonumber \\
	 &\lesssim \| \pi_kf(x) \,  g_p(-i\nabla) \|_{\mathcal{L}^{p,\infty}}^* \, \| \pi_kf(x) g_{p'}(-i\nabla) \|_{\mathcal{L}^{p',\infty}}^*, \label{eq:Holder-for-pi_kf-g}
\end{align}
for any $p\in(2,2+\delta]$, thanks to H\"older's inequality in weak trace ideals (\cite[Theorem 2.1]{
Simon76}).

Since $p>2$, the usual Cwikel estimate, Theorem \ref{thm:Cwikel}, yields
\begin{equation} \label{eq:Cwikel-for-A_pk}
	\| \pi_kf(x) \,  g_p(-i\nabla) \|_{\mathcal{L}^{p,\infty}}^* \lesssim_{\,d} (p-2)^{-\frac{1}{p}}\|\pi_kf\|_{L^p} \|g_p\|_{L^{p,\infty}}^*.
\end{equation}
On the other hand, since $p'<2$, Simon's result, Theorem \ref{thm:simon},  implies
\begin{equation}
	\label{eq:simon res p<2}
	\| \pi_kf(x) \,  g_{p'}(-i\nabla) \|_{\mathcal{L}^{p',\infty}}^* \lesssim_{\,d,\delta} (2-p')^{\frac{1}{p'}-1}\|\pi_kf\|_{\ell^{p'}(L^2)} \|g_{p'}\|_{\ell^{p',\infty}(L^2)}^*.
\end{equation}
It follows from \eqref{eq:Holder-for-pi_kf-g}--\eqref{eq:simon res p<2} that, for all $p>2$,
\begin{align}
\big(\big\| (\pi_kf)(x) \,  g(-i\nabla)\big\|_{\mathcal{L}^{2,\infty}}^*\big)^2 &\lesssim_{\,d,\delta} (p-2)^{-\frac{1}{p}} (2-p')^{\frac{1}{p'}-1} \|\pi_kf\|_{L^p} \|\pi_kf\|_{\ell^{p'}(L^2)} \notag\\
&\qquad\quad\times\inf_{\substack{g_p,g_{p'}\\g^2=g_pg_{p'}}} \big(\|g_p\|_{L^{p,\infty}}^* \|g_{p'}\|_{\ell^{p',\infty}(L^2)}^*\big). \label{eq:newthm1.7}
\end{align}

We now give bounds on the $\pi_k f$ terms. Choosing $p\in(2,2+\delta]$ in such a way that $t_p:= d(\frac12-\frac1p)\le2$, the usual Sobolev embedding, Lemma \ref{lm:sobolev}, together with the quadratic form inequality $\langle-i\nabla\rangle^{t_p}\le\langle\mathbf{h}\rangle^{t_p/2}$ give
\begin{equation*}
\|\pi_kf\|_{L^p}\lesssim_{\,d} \|\langle-i\nabla\rangle^{t_p}\pi_kf\|_{L^2}\lesssim_{\,d}  \Lambda_{k+1}^{t_p/2}\|\pi_kf\|_{L^2}.
\end{equation*}
At this point we take 
\begin{equation*}
\frac{1}{p}=\frac{1}{2}-\frac{\delta}{d\ln\Lambda_{k+1}}
\end{equation*}
so that $p\in(2,2+\delta]$ and $t_p=\frac{\delta}{\ln\Lambda_{k+1}}$, which in turn gives $\Lambda_{k+1}^{t_p/2}= e^{\delta/2}$.

To treat the contribution of $\|\pi_kf\|_{\ell^{p'}(L^2)}$, we use the embedding $L^2(\langle x\rangle^{2r}\mathrm{d}x) \hookrightarrow \ell^{p'}(L^2)$ for any $r>d(\frac{1}{p'}-\frac{1}{2})$, see Lemma \ref{lm:embedding}. With our choice of $p$, we have 
\begin{equation*}
\frac{1}{p'}=\frac{1}{2}+\frac{\delta}{d\ln\Lambda_{k+1}}.
\end{equation*}
 Hence we can choose $r=\frac{2\delta}{\ln\Lambda_{k+1}}\le2$ yielding $\langle x\rangle^{r}\le\langle\mathbf{h}\rangle^{r/2}$ and
\begin{equation*}
	 \|\pi_kf\|_{\ell^{p'}(L^2)} \lesssim_{\,d,\delta} \|\langle x\rangle^r\pi_kf\|_{L^2} \lesssim_{\,d,\delta} \Lambda_{k+1}^{r/2} \|\pi_kf\|_{L^2}.
\end{equation*}
Similarly as before, we observe that $\Lambda_{k+1}^{r/2}= e^\delta$. Hence our previous estimates imply
\begin{equation}\label{eq:prod-pi_k-f}
	\|\pi_kf\|_{L^p}\|\pi_kf\|_{\ell^{p'}(L^2)} \lesssim_{\,d,\delta} \|\pi_kf\|_{L^2}^2.
\end{equation}

Next, using the relations
\begin{equation*}
p-2=\frac{4\delta}{d\ln\Lambda_{k+1}-2\delta} , \quad 2-p'=\frac{4\delta}{d\ln \Lambda_{k+1}+2\delta},
\end{equation*}
yields the bound
\begin{equation}\label{eq:bound-prod-p}
\big [(p-2)(2-p') \big]^{-\frac{1}{p}} \lesssim_{\,d,\delta} \ln \Lambda_{k+1}.
\end{equation}

Putting together \eqref{eq:newthm1.7}, \eqref{eq:prod-pi_k-f} and \eqref{eq:bound-prod-p}
gives
\begin{equation*}
\big\| (\pi_kf)(x) \,  g(-i\nabla)\big\|_{\mathcal{L}^{2,\infty}}^* 
\lesssim_{\,d,\delta} (\ln \Lambda_{k+1})^{\frac12} \|\pi_kf\|_{L^2} \Big( \sup_{\substack{2< p\leq 2+\delta\\\frac{1}{ p}+\frac{1}{ p'}=1}} \inf_{\substack{g_p,g_{p'}\\ g^2=g_pg_{p'}}} \|g_{p}\|_{L^{{ p},\infty}}^* \|g_{p'}\|_{\ell^{{ p}',\infty}(L^2)}^* \Big)^{\frac12}.
\end{equation*}
Since $\ln \Lambda_{k+1}=e\ln \Lambda_{k}$, the $k$-dependent part of the right hand side can then be summed over $k$ as follows:
\begin{align*}
 \sum_{k\in\mathbb{N}}  (\ln \Lambda_{k+1})^{\frac12} \|\pi_kf\|_{L^2}
 & \lesssim\sum_{k\in\mathbb{N}} \big\|\pi_k((\ln\mathbf{h})^{\frac12}f)\big\|_{L^2}  \\
&\lesssim_{\,d,\varepsilon} \sum_{k\in\mathbb{N}} k^{-\frac12-\varepsilon} \big\|\pi_k((\ln\mathbf{h})^{\frac12} (\ln\ln\mathbf{h})^{\frac12+\varepsilon} f)\big\|_{L^2} \\
 &\lesssim_{\,d,\varepsilon} \big\| (\ln\mathbf{h})^{\frac12} (\ln\ln\mathbf{h})^{\frac12+\varepsilon}f\big\|_{L^2},
\end{align*}
where we used that 
$0<\varepsilon%<\frac12
$ and the Cauchy-Schwarz inequality in the last inequality. This along with 
\eqref{eq:double_ln} implies the statement of Theorem~\ref{th: norm S^2,infty}.
\end{proof}

\section{Proof of Theorems~\ref{th: N_<0} and \ref{th: N_<0critical}}\label{sec:main_thm}
In this section we prove Theorem~\ref{th: N_<0} using the Birman-Schwinger principle, the variational principle, Theorem~\ref{thm:low} and Theorem~\ref{thm:high}.
\begin{proof}[Proof of Theorems~\ref{th: N_<0} and \ref{th: N_<0critical}]
Let $E<0$. To estimate $N_{\le E}(H_s)$ we use the Birman-Schwinger principle (see Proposition \ref{prop:birm-sch}) which shows that
\begin{equation}
	\label{comp:NH-NKE}
N_{\le E}(H_s) = N_{\ge 1}(K_E),
\end{equation}
where we recall that the Birman-Schwinger operator $K_E$ is given by
\begin{equation*}
K_E = v(x) \big( (-\Delta)^s - E \big )^{-1} v(x), 
\end{equation*}
with $v(x)= \sqrt{V(x)}$. We recall also that $n=\lfloor s-\frac{d}{2}\rfloor$ and
\begin{equation}\label{eq:F_n2}
	\mathcal{F}_n:=\mathrm{span}\big\{ x^\alpha v \, | \, \alpha\in\mathbb{N}_0^d , \, |\alpha|\le n \big\},
\end{equation}
where $|\alpha|=\sum_{j=1}^d \alpha_j$ and $x^\alpha = \prod_{j=1}^d x_j^{\alpha_j}$. 
Note that, with $S_1:=\{\alpha\in\mathbb{N}_0^d\mid |\alpha|\leq n\}$,
\begin{align}\label{eq:dimFn}
	\mathrm{dim}(\mathcal{F}_n)\le  |S_1| = \binom{d+n}{d} .
\end{align}
Indeed, set $S_2:=\{X\subseteq \{0,1,2,\dots,1,d+n-1\} \mid |X|=d \}$.
Then
\[ S_1\ni \alpha \mapsto \{ k-1 +  \alpha_1 +\cdots + \alpha_k \mid 1\leq k\leq d\}\in S_2\]
and
\[S_2\ni \{\beta_1<\cdots<\beta_d\} \mapsto (\beta_1,\beta_2-\beta_1-1,\dots,\beta_d-\beta_{d-1}-1) \in S_1 \]
are inverse functions of each other and hence bijections. It follows that $|S_1|=|S_2|=\binom{d+n}{d}$.

%Indeed, we have 
%\begin{equation}
%	\label{eq: alpha=m}
%	|\{\alpha\in\mathbb{N}_0^d\mid |\alpha|\leq n\}|=\sum_{m=0}^n|\{\alpha\in\mathbb{N}_0^d\mid |\alpha|=m\}| =\sum_{m=0}^n\binom{d+m-1}{m}
%\end{equation}
%since to obtain $\alpha$ of length $m$ we need to place $m$ ones within the $d$ components of $\alpha$. This means we need to choose $m$ positions among the $d$ available ones, possibly with repetition (when we have $\alpha_k>1$ for some $k =1,\ldots,d$), hence the cardinality in \eqref{eq: alpha=m} is the number of choices of $m$ elements among $d$ with repetition. To obtain \eqref{eq:dimFn} we just remark that $\sum_{m=0}^n \binom{d+m-1}{m} = \binom{d+n}{d}$ which can be proved by induction. 

By the variational principle recalled in Proposition~\ref{prop: var princ}, if $\Pi_{\mathcal{F}_n}^\perp$ denotes the orthogonal projection onto $\mathcal{F}_n^\perp$, we have 
\begin{equation}\label{eq:variationnal-KE}
	N_{\ge 1}(K_E) \le \binom{d+n}{d}  + N_{\ge1}( K^\perp_E),
\end{equation}
with $K_E^\perp=\Pi_{\mathcal{F}_n}^\perp K_{E,}\Pi_{\mathcal{F}_n}^\perp$.

Let $j_{\max}:=\max\{j\geq0\mid\lambda_{j}(K_{E}^{\perp})\geq1\}$. Using
that $\lambda_{j}(K_{E}^{\perp})$ is a decreasing sequence and actually
coincides with the singular values of $K_{E}^{\perp}$ (as $K_{E}^{\perp}\geq0$), we have
\begin{equation}\label{eq:NKE<KEL1infty}
N_{\geq1}(K_{E}^{\perp})=(j_{\max}+1)\leq(j_{\max}+1)\lambda_{j_{\max}}(K_{E}^{\perp})\leq\|K_{E}^{\perp}\|_{\mathcal{L}^{1,\infty}}^{*}\,.
\end{equation}
We now use the decomposition in low- and high-frequencies parts of
$K_{E}^{\perp}$ as defined in \eqref{def: K_0 + K_inf}--\eqref{def: K_0 + K_inf2}, obtaining
\begin{equation}\label{eq:KE-leq-KElow+KEhigh}
\|K_{E}^{\perp}\|_{\mathcal{L}^{1,\infty}}^{*}\leq2\|K_{E,<}^{\perp}\|_{\mathcal{L}^{1,\infty}}^{*}+2\|K_{E,>}^{\perp}\|_{\mathcal{L}^{1,\infty}}^{*}\leq2\|K_{E,<}^{\perp}\|_{\mathcal{L}^{1}}+2\|K_{E,>}\|_{\mathcal{L}^{1,\infty}}^{*}.
\end{equation}
It follows from \eqref{comp:NH-NKE}-\eqref{eq:KE-leq-KElow+KEhigh}, Theorem~\ref{thm:low} and Theorem~\ref{thm:high} that
\begin{equation*}
	N_{<0}(H_s) -\binom{d+n}{d} \lesssim_{\,d,s}
	\begin{cases}
%			C_{d,p} \big\|\sqrt{1+ \ln\langle x \rangle} \, v\big\|_{L^2} \big\|\sqrt{1+ \ln\langle x \rangle} \, v\big\|_{\ell^2(L^p)} & \textnormal{if }s= \frac{d}{2}\\
			 \big\|\langle x\rangle^{s-\frac{d}{2}} \, v \big\|_{L^2}^2 & \textnormal{if } s-\frac{d}{2} \notin \N_0 ,\vspace{0,1cm}\\
			 \big\|\langle x \rangle ^{s-\frac{d}{2}}\sqrt{1+ \ln\langle x \rangle} \, v\big\|_{L^2}^2 & \textnormal{if } s-\frac{d}{2} \in \N , \vspace{0,1cm} \\
			 \big \|(\ln \mathbf{h})^{\frac12}(\ln\ln \mathbf{h})^{\frac12+\varepsilon}\,v\big\|^2_{L^2} & \textnormal{if } s=\frac{d}{2}.
	\end{cases}
\end{equation*}
This proves Theorem~\ref{th: N_<0} in the case where $s-d/2\in\N$ and Theorem \ref{th: N_<0critical}. In the case where $s-d/2\notin\N_0$, it remains to show that we can replace $\langle x\rangle$ by $|x|$. To this end, we argue as follows.\footnote{We are grateful to R. Frank for pointing out this argument to us.} By scaling, the operators $H_s$ and $R^{2s}(-\Delta)^s-V(R^{-1}x)$ are unitarily equivalent, for any $R>0$. Hence
\begin{equation*}
N_{<0}(H_s)=N_{<0}\big( (-\Delta)^s-R^{-2s}V(R^{-1}x)\big).
\end{equation*}
Applying the previous estimate (for $s-d/2\notin\N_0$), we obtain that
\begin{equation*}
	N_{<0}(H_s) -\binom{d+n}{d} \lesssim_{\,d,s}
			 R^{-2s} \int_{\mathbb{R}^d} \langle x\rangle^{2s-d} V(R^{-1}x) \mathrm{d}x =  \int_{\mathbb{R}^d} (R^{-2}+x^2)^{s-\frac{d}{2}} V(x) \mathrm{d}x .
\end{equation*}
Letting $R\to\infty$, using the monotone convergence theorem, we deduce that
\begin{equation*}
	N_{<0}(H_s) -\binom{d+n}{d} \lesssim_{\,d,s} \int_{\mathbb{R}^d} |x|^{2s-d} V(x) \mathrm{d}x ,
\end{equation*}
which proves Theorem \ref{th: N_<0} in the case where $s-d/2\notin\N_0$.
\end{proof}

%We mention that, using Theorem~\ref{th: norm S^2,infty2} instead of Theorem~\ref{th: norm S^2,infty} to bound the high-energies part, the same proof gives the following estimate which is different from that of Theorem~\ref{th: N_<0} in the critical case $s=d/2$.
%%
%%
%\begin{theorem}
%		\label{th: N_<0bis}
%	Let $d\geq1$, $p>2$, $\frac{1}{p'}=1-\frac{1}{p}$. There exist $C_{d,p}>0$ such that 
%\begin{equation*}
%	N_{<0}\big((-\Delta)^{\frac{d}{2}}-V(x)\big) \leq 1 + 
%			C_{d,p} \big(\big\|\sqrt{1+ \ln\langle x \rangle} \, v\big\|_{L^2}^2+ \|v\|_{\ell^{p'}(L^2)} \|v\|_{L^p} \big),
%	\end{equation*}
%for all $v$ such that the right hand side is finite. 
%\end{theorem}
%
%
%Recall that $\mathcal{F}_n=\mathrm{span}\big\{ x^\alpha v \, | \, \alpha\in\mathbb{N}_0^d , \, |\alpha|\le n \big\}$, where $v=V^\frac12$. 
We conclude this section with a proposition showing that $H_s=(-\Delta)^s-V$ has at least $\mathrm{dim}\,\mathcal{F}_n$ negative eigenvalues for smooth compactly supported $V$ (with $\mathcal{F}_n$ defined in \eqref{eq:F_n2}). Taking $V$ such that $\mathrm{dim}\,\mathcal{F}_n$ is maximal, i.e. $\mathrm{dim}(\mathcal{F}_n)= |\{\alpha\in\mathbb{N}_0^d\mid |\alpha|\leq n\}|$, shows that the constant $\binom{d+n}{d}$ cannot be removed from the statement of Theorem~\ref{th: N_<0}. The proof of  Proposition \ref{prop:binom_constant} is a fairly direct generalization of that given in \cite[Theorem XIII.11]{ReedSimon4}.
\begin{proposition}\label{prop:binom_constant}
Let $d\ge 1$, $s\ge\frac{d}{2}$ and $n=\lfloor s-\frac{d}{2}\rfloor$. Let $V\in\mathrm{C}_0^\infty(\mathbb{R}^d)$ be such that $V\ge0$. Then the operator $H_s=(-\Delta)^s-V$ has at least $\mathrm{dim}\, \mathcal{F}_n$ negative eigenvalues.
\end{proposition}
\begin{proof}
By the Birman-Schwinger principle (see Proposition \ref{prop:birm-sch}), it suffices to show that $N_{\ge1}(K_E)\ge\mathrm{dim}\,\mathcal{F}_n$ for $E<0$, $|E|$ small enough, where $K_E$ is the Birman-Schwinger operator defined as above, namely $K_E=v(x)((-\Delta)^s - E)^{-1}v(x)$.

Let $\varphi\in\mathcal{F}_n$, $\varphi\neq0$. Then $\varphi\in\mathrm{C}_0^\infty(\mathbb{R}^d)$ and we claim that there exist $\varepsilon>0$ and $c>0$ (which depends on $V$, $\varphi$, $n$) such that, for all $\xi\in\mathbb{R}^d$ with $|\xi|\le\varepsilon$, 
\begin{equation}\label{eq:propphi}
|\widehat{v\varphi}(\xi)|\ge c|\xi|^{n}.
\end{equation}
Indeed, if this property did not hold, then we would have that for all $\alpha\in\mathbb{N}_0^d$ such that $|\alpha|\le n$, 
\begin{equation*}
0=\partial_\xi^\alpha\widehat{v\varphi}(0)=(-i)^\alpha\widehat{x^\alpha v\varphi}(0)=(-i)^\alpha\int_{\mathbb{R}^d}x^\alpha v(x)\varphi(x)\mathrm{d}x,
\end{equation*}
which contradicts the facts that $\varphi\in\mathcal{F}_n$ and $\varphi\neq0$.

Now using \eqref{eq:propphi}, we write, for all $\varphi\in\mathcal{F}_n$, $\varphi\neq0$,
\begin{equation*}
\langle \varphi,K_E\, \varphi\rangle = \int_{\mathbb{R}^d} \big(|\xi|^{2s}-E\big)^{-1} \big| \widehat{v\varphi}(\xi)\big|^2\mathrm{d}\xi \ge c \int_{|\xi|\le\varepsilon} |\xi|^{2n} \big(|\xi|^{2s}-E\big)^{-1} \mathrm{d}\xi.
\end{equation*}
Since $2s-2n\ge d$, the previous integral tends to infinity as $E\to0$. Hence it follows from the min-max principle (see Theorem \ref{thm:min-max-principle}) that, for $|E|$ small enough, $K_E$ has at least $\mathrm{dim} \, \mathcal{F}_n$ eigenvalues larger than $1$. This concludes the proof.
\end{proof}

\appendix

\section{Variational principle} \label{app:variational}
In this appendix, we recall how to estimate the number of eigenvalues larger than $1$ of an operator, by the number of eigenvalues larger than $1$ of the restriction of this operator to a linear subspace, up to the dimension of the subspace itself. We refer to e.g. \cite[Section 1.2.3]{FrankLaptevWeidl23} for general versions of the variational principle.

If $\mathcal{F}$ is a closed linear subspace of a Hilbert space $\mathcal{H}$, $\Pi_{\mathcal{F}}$ denotes the orthogonal projection onto $\mathcal{F}$.
\begin{proposition}\label{prop: var princ}
Let $K$ a compact self-adjoint non-negative operator on a Hilbert space $\mathcal{H}$.
Then, for any linear subspace $\mathcal{F}$  of $\mathcal{H}$ of finite dimension,
\begin{equation}
N_{\geq1}(K)\leq \dim \mathcal{F}+N_{\geq1}(\Pi_{\mathcal{F}^{\perp}}K\Pi_{\mathcal{F}^{\perp}})\,.\label{eq:bound-number-eigenvalues-greater-than-one-projection}
\end{equation}
\end{proposition}

To prove this result we use the following simple version of the min-max principle. (See e.g. \cite{ReedSimon4} for a more general version.)

\begin{theorem}\label{thm:min-max-principle}
	Let $K$ a compact selfadjoint non-negative operator on a Hilbert space $\mathcal{H}$.  Then the sequence defined for $j\ge0$ by
\begin{equation*}
		\lambda_j(K) = \min_{\dim \mathcal{S}=j} \max_{\substack{u \in \mathcal{S}^\perp\\ \|u\|=1}}  \langle u , K u\rangle
\end{equation*}
coincides with the  non-increasing sequence either of the positive eigenvalues of $K$ if $K$ is of infinite rank, or, otherwise, of all its eigenvalues. Here the minimum is taken over all linear subspaces $\mathcal{S}$ of $\mathcal{H}$ of dimension $\dim\mathcal{S}=j$.
\end{theorem}

\begin{proof}[Proof of Proposition~\ref{prop: var princ}]
Let $D=\dim\mathcal{F}$. By the min-max principle in Theorem~\ref{thm:min-max-principle},
\[
\lambda_{D+k}(K)=\min_{\dim \mathcal{S}=D+k}\max_{\substack{u\in \mathcal{S}^{\perp}\\ \|u\|=1}} \langle u,Ku\rangle\,.
\]
For any subspace $\mathcal{V}$ of $\mathcal{F}^{\perp}$ of dimension $k$, we have $\dim (\mathcal{F}\oplus\mathcal{V})= D+k$. 
As $u=\Pi_{\mathcal{F}^{\perp}}u$ for $u\in\mathcal{F}^{\perp}$,
\[
\lambda_{D+k}(K)
\leq 
\max_{\substack{u\in (\mathcal{F}\oplus\mathcal{V})^{\perp}\\
\|u\|=1
}
} \langle u,Ku\rangle
=  \max_{\substack{u\in\mathcal{F}^{\perp}\cap \mathcal{V}^{\perp}\\ \|u\|=1}} 
\langle \Pi_{\mathcal{F}^{\perp}}u,K\Pi_{\mathcal{F}^{\perp}}u\rangle
= \max_{\substack{u\in\mathcal{V}^{\perp}\\ \|u\|=1}} 
\langle u,\Pi_{\mathcal{F}^{\perp}}K\Pi_{\mathcal{F}^{\perp}}u\rangle,
\]
for any subspace $\mathcal{V}$ of $\mathcal{F}^{\perp}$ of dimension $k$.
This implies
\begin{equation}
\lambda_{D+k}(K)\leq\min_{\substack{\dim\mathcal{V}=k\\
\mathcal{V}\subseteq\mathcal{F}^{\perp}
}
}\max_{\substack{u\in \mathcal{V}^{\perp}\\
\|u\|=1
}
} \langle u,\Pi_{\mathcal{F}^{\perp}}K\Pi_{\mathcal{F}^{\perp}}u\rangle =\lambda_{k}(\Pi_{\mathcal{F}^{\perp}}K\Pi_{\mathcal{F}^{\perp}})\,.\label{eq:comparison-eigenvalues-K-PKP}
\end{equation}
As eigenvalues given by the min-max principle are sorted in non-increasing
order, we deduce that
\begin{equation*}
N_{\geq1}(K)-D  \leq |\{k\geq0 \mid\lambda_{D+k}(K)\geq1\}| 
  \leq |\{k\geq0 \mid\lambda_{k}(\Pi_{\mathcal{F}^{\perp}}K\Pi_{\mathcal{F}^{\perp}})\geq1\}| 
  = N_{\geq1}(\Pi_{\mathcal{F}^{\perp}}K\Pi_{\mathcal{F}^{\perp}}),
\end{equation*}
which yields \eqref{eq:bound-number-eigenvalues-greater-than-one-projection}.
\end{proof}

\section{Birman-Schwinger principle}
\label{app: Birman-S}
In this section, for the convenience of the reader, we recall a proof of the Birman-Schwinger principle for $H_s = (-\Delta)^s - v^2$ and $K_E := v \big( (-\Delta)^s - E \big )^{-1} v$, under the following assumptions:
\begin{assumption}\label{hyp:v-high-frequency} 
Let $d\geq 1$, $s\geq d/2$,  and $v$ measurable and real-valued, such that
\begin{itemize}
	\item either $v \in L^{2}$, if $s>d/2$,
	\item or $v \in \mathcal{D}( (\ln\mathbf{h})^{\frac12} (\ln\ln\mathbf{h})^{\frac12+\varepsilon} )\subset L^2$ for some $\varepsilon>0$ if $s=d/2$.
\end{itemize}
\end{assumption}
Here we denote by $\mathcal{D}(A)$ the domain of an operator $A$. We refer to e.g. \cite[Section 1.2.8]{FrankLaptevWeidl23} for a proof of the Birman-Schwinger principle in a general abstract setting.
\begin{remark}\label{rk:embeddings}
 Hypothesis~\ref{hyp:v-high-frequency} and $E<0$ ensure that the chain 
 \begin{align}\label{chain:s}
	L^2\xrightarrow{\ \ v\times\ \ } L^1\hookrightarrow (L^\infty)^* \hookrightarrow H^{-s} \xrightarrow{((-\Delta)^s-E)^{-1}} H^s\hookrightarrow L^\infty \xrightarrow{\ \ v\times\ \ } L^2,
\end{align}
holds for $s>d/2$. For $s=d/2$, we observe that, for all $p>2$, $\frac1p+\frac{1}{p'}=1$,
\begin{equation*}
(|\xi|^d-E)^{-1}=(|\xi|^d-E)^{-\frac{1}{p}} (|\xi|^d-E)^{-\frac{1}{p'}},
\end{equation*}
with 
\begin{equation*}
\|(|\xi|^d-E)^{-\frac{1}{p}}\|^*_{L^{p,\infty}}\lesssim_{\,d}1 \quad\text{ and } \quad\|(|\xi|^d-E)^{-\frac{1}{p'}}\|^*_{\ell^{p',\infty}(L^2)}\lesssim_{\,d} 1,
\end{equation*}
 uniformly in $p\in(2,2+\delta]$ for any $\delta>0$ (this follows from a similar calculation as in \eqref{eq:comput_gp}--\eqref{eq:comput_gp'}). Theorem \ref{th: norm S^2,infty} then shows that $((-\Delta)^{d/2}-E)^{-\frac12}v(x)$ belongs to $\mathcal{L}^{2,\infty}$ and hence is bounded. Its adjoint is then also bounded. This shows that the operator of multiplication by $v$ is bounded from $H^{d/2}$ to $L^2$ and from $L^2$ to $H^{-d/2}$. Therefore the chain
\begin{align}\label{chain:d/2}
	L^2\xrightarrow{\ \ v\times\ \ } H^{-d/2} \xrightarrow{((-\Delta)^{d/2}-E)^{-1}} H^{d/2} \xrightarrow{\ \ v\times\ \ }L^2,
\end{align}
holds. In particular $K_{E}$ is a bounded operator on $L^2$.

Moreover $K_E$ is also compact. For $s>d/2$ it is Hilbert-Schmidt, since its integral kernel is given by the $L^2$ function
\begin{equation*}
	-v(x) v(y) \int e^{-i(x-y)\xi} \frac{1}{|\xi|^{2s}- E} \mathrm{d}\xi\, . 
\end{equation*}
For $s=d/2$, this follows again from Theorem \ref{th: norm S^2,infty}. 

Note that this also implies that $H_s$ is self-adjoint by the KLMN theorem \cite[Theorem X.17]{ReedSimon2} and that the essential spectrum of $H_s$ is equal to $[0,\infty)$ thanks to Weyl's essential spectrum theorem \cite[Theorem XIII.14]{ReedSimon4} (see also \cite[Section XIII.4, Example 7]{ReedSimon4}).
\end{remark}

Recall that $N_{\le r}(A)$ (respectively $N_{\geq r}(A)$) denotes the number of eigenvalues less or equal (respectively larger or equal) than $r$ of a self-adjoint operator $A$, counted with multiplicity.
\begin{proposition}[Birman-Schwinger principle]\label{prop:birm-sch}
Assume  $E<0$ and Hypothesis~\ref{hyp:v-high-frequency} holds. Then
\begin{equation}\label{eq:bound-Birman-Schwinger}
	N_{\le E}(H_s) = N_{\ge 1}(K_E)\,.
\end{equation}
\end{proposition}

The non-increasing sequence of eigenvalues $(\lambda_j(K))_{j\ge0}$ is rigorously defined in the statement of Theorem \ref{thm:min-max-principle}. We prove Proposition~\ref{prop:birm-sch} following the arguments of \cite{lieb_seiringer}, using properties of the maps $E\mapsto\lambda_j(K_E)$ that we collect in the following lemma.

\begin{lemma}\label{lemma: eval KE}
Assume Hypothesis~\ref{hyp:v-high-frequency} holds. For any $j\geq0$, the map $E\mapsto \lambda_j (K_E)$ is non-decreasing, continuous on~$(-\infty,0)$ and goes to 0 as $E\to-\infty$. 
\end{lemma}
\begin{proof}
	As $x\mapsto 1/x$ is operator monotone, the expression of $\lambda_j(K_E)$ given in the min-max principle (Theorem \ref{thm:min-max-principle}) yields that $E\mapsto \lambda_j (K_E)$ is non-decreasing. 
	
	To prove continuity we use first the resolvent identity: let $E'<E<0$. Then
	\begin{align*}
		K_E - K_{E'} = (E-E') \, v \, \big( (-\Delta)^s - E \big )^{-1}\big( (-\Delta)^s - E' \big )^{-1} \, v \leq \frac{E-E' }{-E'}K_E,
	\end{align*}
	and hence, for all $u\in L^2$,
	\begin{align*}
		\langle K_E\, u, u\rangle
		\leq 	\langle K_{E'}\, u, u\rangle  + \frac{E-E' }{-E'}	\|K_E\|_{\mathcal{L}^\infty} \|u\|_{L^2}^2.
	\end{align*}
 The min-max principle (Theorem \ref{thm:min-max-principle}) and the previous inequality then yield
	\begin{align*}
		\lambda_j (K_E ) \leq \max_{\substack{u \in \mathcal{S}^\perp\\ \|u\|_{L^2 }=1}} 	\langle K_E \, u, u\rangle  \leq \max_{\substack{u \in \mathcal{S}^\perp\\  \|u\|_{L^2 }=1}}	\langle K_{E'} \, u, u\rangle + \frac{E-E' }{-E'}	\|K_E\|_{\mathcal{L}^\infty}, 
	\end{align*}
	for all subspace $\mathcal{S}$ of $L^2$ of dimension $j$. Hence, taking the minimum over all such spaces and using again the min-max principle, we obtain
	\begin{align*}
		\lambda_j (K_E ) \leq  \lambda_j (K_{E'} ) + \frac{E-E'}{-E'}	\|K_E\|_{\mathcal{L}^\infty} .
	\end{align*}
	Together with  $\lambda_j (K_{E'} )\leq\lambda_j (K_E )$, this gives the continuity with respect to $E$ of $\lambda_j (K_E)$. 
	
To prove that $\lambda_j(K_E)\to0$ as $E\to-\infty$, since $\lambda_j(K_E)\le\|K_E\|_{\mathcal{L}^\infty}$, it suffices to show that $\|K_E\|_{\mathcal{L}^\infty}=\|[(-\Delta)^s-E]^{-1/2}v(x)\|^2_{\mathcal{L}^\infty}\to0$.

Suppose that $s=d/2$. Recall that $v\in\mathcal{D}((\ln\mathbf{h})^{\frac12} (\ln\ln\mathbf{h})^{\frac12+\varepsilon})$ by assumption. Let $\varepsilon'>0$ and let $R_{\varepsilon'}>0$ be such that
\begin{equation}\label{eq:appA1}
\big\|\mathds{1}_{\mathbf{h}\ge R_{\varepsilon'}} (\ln\mathbf{h})^{\frac12} (\ln\ln\mathbf{h})^{\frac12+\varepsilon} v\big\|_{L^2}\le\varepsilon'.
\end{equation}
Setting $v_{\varepsilon'}:=\mathds{1}_{\mathbf{h}< R_{\varepsilon'}}  v$, we have $v_{\varepsilon'}\in L^\infty$ (as $v_{\varepsilon'}$ is a finite linear combination of bound states of $\mathbf{h}$) and therefore we can write %\footnote{Viviana: it should be corrected to have a $(-E)^{-\frac12}$ also in the second term? }
\begin{align}
\big\|[(-\Delta)^s-E]^{-\frac12}v(x)\big\|_{\mathcal{L}^\infty}&\le\big\|[(-\Delta)^s-E]^{-\frac12}v_{\varepsilon'}(x)\big\|_{\mathcal{L}^\infty}+\big\|[(-\Delta)^s-E]^{-\frac12}(v(x)-v_{\varepsilon'}(x))\big\|_{\mathcal{L}^\infty} \notag\\
&\le(-E)^{-\frac12}\|v_{\varepsilon'}\|_{L^\infty}+C \big\| (\ln\mathbf{h})^{\frac12} (\ln\ln\mathbf{h})^{\frac12+\varepsilon}(v-v_{\varepsilon'})\big\|_{L^2}, \label{eq:appA2}
\end{align}
for some $C>0$, uniformly in $E\le-1$. In the second inequality, we used that $\|A\|_{\mathcal{L}^\infty}\lesssim\|A\|_{\mathcal{L}^{2,\infty}}$, for any operator $A\in\mathcal{L}^{2,\infty}$, together with Theorem \ref{th: norm S^2,infty} (applied with $g(\xi)=[|\xi|^{2s}-E]^{-1/2}$, so that $g^2$ can be decomposed as $g^2(\xi)=g_p(\xi)g_{p'}(\xi)$ with $p>2$, $\frac{1}{p}+\frac{1}{p'}=1$ and $g_p(\xi)=[|\xi|^{2s}-E]^{-1/p}$; a similar calculation as in \eqref{eq:comput_gp}--\eqref{eq:comput_gp'} then shows that $\|g_p\|_{L^{p,\infty}}$, $\|g_{p'}\|_{\ell^{p',\infty}(L^2)}$ are uniformly bounded in $p\in(2,2+\delta]$ and $E\le-1$ for any $\delta>0$). Combining \eqref{eq:appA1} and \eqref{eq:appA2} shows that $\|[(-\Delta)^s-E]^{-1/2}v(x)\|_{\mathcal{L}^\infty}\to 0$ as $E\to-\infty$.

In the case where $s>d/2$, it suffices to write
	\begin{multline*}
	\|K_E\|_{\mathcal{L}^\infty}=\sup_{\|u\|_{L^2}=1}\langle u,K_Eu\rangle
	=\sup_{\|u\|_{L^2}=1}\int_{\mathbb{R}^d}\big(|\xi|^{2s}-E\big)^{-1}|\widehat{vu}(\xi)|^2\mathrm{d}\xi\\
	\le\big\|\big(|\xi|^{2s}-E\big)^{-1}\big\|_{L^1} \sup_{\|u\|_{L^2}=1}\big\| |\widehat{vu}|^2\big\|_{L^{\infty}}.
	\end{multline*}
Now we have 
	\begin{equation*}
	\big\| |\widehat{vu}|^2\big\|_{L^{\infty}}=\big\|\widehat{vu}\big\|^2_{L^{\infty}}\lesssim_{\,d} \|u\|_{L^2}^2\|v\|_{L^2}^2,
	\end{equation*}
	and the dominated convergence theorem shows that $\|(|\xi|^{2s}-E)^{-1}\|_{L^1}\to0$ as $E\to-\infty$. This concludes the proof.
\end{proof}

Now we are ready to prove Proposition \ref{prop:birm-sch}.

\begin{proof}[Proof of Proposition \ref{prop:birm-sch}]
Any eigenfunction $\psi$ of $H_s$ associated to an eigenvalue $E'<0$ is in particular in the domain of~$H_s$ (hence in $H^s$, the form domain of~$H_s$), and satisfies
\begin{equation*}
	((-\Delta)^s - E') \psi = v^2\,\psi.
\end{equation*}
We set $\phi  = v\psi\in H^{-s}$ (see Remark \ref{rk:embeddings}).
The resolvent  $((-\Delta)^s - E')^{-1}$ applied to the equality above yields
\begin{equation*}
	   \psi = \big( (-\Delta)^s - E' \big )^{-1} v \phi\in H^s\,,
\end{equation*}
which in turn implies that $\phi\neq 0$. Multiplying by $v$ then gives
\begin{equation*}
	  \phi = v \big( (-\Delta)^s - E' \big )^{-1} v \phi\in L^2,
\end{equation*}
so that $\phi$ is an eigenvector of $K_{E'}$ corresponding to the eigenvalue $1$.

Viceversa, for any eigenfunction $\phi\in L^2$ of $K_{E'}$ associated to the eigenvalue $1$,
we set $\psi = ((-\Delta)^s - E')^{-1}v \phi\in H^s\subset L^2$. Multiplying by $v$ yields $v\psi = \phi\neq 0 $, so that $\psi\neq0$ and
\begin{equation*}
	((-\Delta)^s - E') \psi = v \phi = v^2  ((-\Delta)^s -E')^{-1}v \phi = v                                                                                                                                                                                                                                  ^2 \psi\,.
\end{equation*}
It follows that $\psi$ is an eigenvector of $H_s$ associated to the eigenvalue $E'$. 

We have thus, for any $E'<0$, a bijection between the eigenfunctions $\phi$ of $K_{E'}$ corresponding to the eigenvalue $1$, and the eigenfunctions $\psi $ of $H_s$ corresponding to $E'$.
Hence
\begin{align}
N_{\leq E}(H_s) & =\sum_{E'\leq E}\dim\ker(H_s-E')
=\sum_{E'\leq E}|\{j\mid\lambda_{j}(K_{E'})=1\}|\,.\label{eq:N-H-leq-E-as-a-sum}
\end{align}
Now, for every $j$, the map $E\mapsto\lambda_{j}(K_{E})$ takes at
most once the value~$1$, because otherwise the set of eigenvalues of $H_s$ would contain an interval $[E_1,E_2]\subset (-\infty, 0)$, which is impossible. 
It follows that
\begin{equation}
\sum_{E'\leq E}|\{j\mid\lambda_{j}(K_{E'})=1\}|=|\{j\mid\exists E'\leq E,\lambda_{j}(K_{E'})=1\}|\,.\label{eq:lambda-j-take-only-once-the-value-one}
\end{equation}
As, for any $j$, $E'\mapsto\lambda_{j}(K_{E'})$ is continuous and $\lambda_{j}(K_{E'})\to0$ as $E'\to-\infty$, we deduce that
\begin{equation}
|\{j\mid\exists E'\leq E,\lambda_{j}(K_{E'})=1\}|=|\{j\mid\lambda_{j}(K_{E})\geq1\}|=N_{\geq1}(K_{E})\,.\label{eq:lambda-j-go-to-0-continuously-at-minus-infinity}
\end{equation}
The bound \eqref{eq:bound-Birman-Schwinger} then follows from~(\ref{eq:N-H-leq-E-as-a-sum}), (\ref{eq:lambda-j-take-only-once-the-value-one})
and (\ref{eq:lambda-j-go-to-0-continuously-at-minus-infinity}).
\end{proof}

\section{Proofs of Theorem \ref{thm:simon} and Lemma \ref{lm:embedding}}\label{app:p'}
In this section we prove Theorem \ref{thm:simon} and Lemma \ref{lm:embedding} which were used in the proof of Theorem~\ref{th: norm S^2,infty}. To obtain Theorem \ref{thm:simon}, we reproduce the proof of \cite[Theorem 4.6]{Simon05}, carefully following the dependence on the parameter $p'$ in all the estimates.

%We now discuss the constants in inequality \eqref{eq:simon res p<2} and the embedding $L^2(\langle x\rangle^{2r}\mathrm{d}x) \hookrightarrow \ell^{p'}(L^2)$. For the convenience of the reader, we include the proof of the result in \cite[Theorem 4.6]{Simon05} giving the explicit constant in the inequality. 
%\begin{lemma}[Theorem 4.6, \cite{Simon05}]
%	Let $1<p' \leq 2$, then 
%	\begin{equation*}
%		\|f(x) g(-i\nabla)\|_{\mathcal{L}^{p',\infty}}^* \leq \ldots \|f\|_{\ell^{p'}(L^2)} \|g\|_{\ell^{p',\infty}(L^2)}^*. 
%	\end{equation*}
%In particular, the constant can be choosen uniform in the limit $ p'\to 2$. 
%%In particular, the constant is uniformly bounded in the limit $p'\to 2$.
%\end{lemma}

\begin{proof}[Proof of Theorem \ref{thm:simon}]
%\textcolor{blue}{Throughout the proof of this theorem, $a\lesssim b$ means that there exists $C_{d,\delta'}>0$ such that $a\leq C_{d,\delta'}b$, and this constant may change from one line to the other.}
		
	Assume that $\|f\|_{\ell^{p'}(L^2)}= \|g\|_{\ell^{p',\infty}(L^2)}^*= 1$. Recall that $\chi_\mathbf{m}$ stands for the characteristic function of the unit hypercube of $\R^d$ with center $\mathbf{m}\in\mathbb{Z}^d$ and, for all function $f:\mathbb{R}^d\to\mathbb{C}$, $f_\mathbf{m}:=\chi_\mathbf{m}f$. We set $\tilde{f}_{\mathbf{m}} := \frac{f_{\mathbf{m}}}{\|f_{\mathbf{m}}\|_{L^2}}, \tilde{g}_{\mathbf{m}} := \frac{g_{\mathbf{m}}}{\|g_{\mathbf{m}}\|_{L^2}}$ and write 
	\begin{equation*}
		f = \sum_{\mathbf{m}\in\Z^d} a_{\mathbf{m}}\tilde{f}_{\mathbf{m}},\quad  a_{\mathbf{m}}:= \|f_{\mathbf{m}}\|_{L^2},\quad	g = \sum_{\mathbf{m}\in\Z^d} b_{\mathbf{m}}\tilde{g}_{\mathbf{m}},\quad  b_{\mathbf{m}}:= \|g_{\mathbf{m}}\|_{L^2},
\end{equation*}
so that $ \|a_{\mathbf{m}}\|_{\ell^{p'}} = \|b_{\mathbf{m}}\|_{\ell^{p', \infty}}^*=1$. As in \cite[Theorem 4.6]{Simon05}, for any $n \in \Z$, we define 
\begin{eqnarray*}
	f_n& :=& \sum_{2^{n-1}<a_{\mathbf{m}}\leq 2^n} a_{\mathbf{m}}\tilde{f}_{\mathbf{m}}, \qquad g_n := \sum_{2^{n-1}<b_{\mathbf{m}}\leq 2^n} b_{\mathbf{m}}\tilde{g}_{\mathbf{m}}\\
	A_n&:= & \sum_{l+k\leq n} f_{l}(x) g_k(-i\nabla), \qquad 
	B_n:=  \sum_{l+k> n} f_{l}(x) g_k(-i\nabla),
\end{eqnarray*} 
so that $f(x) g(-i\nabla) = A_n + B_n$. Then using Fan's inequality \cite[Theorem 1.7]{Simon05}:
\begin{equation}
	\label{eq: mu_m uneven}
	\mu_{m}(f(x) g(-i\nabla)) \leq \mu_{m/2 +1/2}(A_n) + \mu_{m/2 +1/2}(B_n), \quad m \textnormal{ odd},
\end{equation}
and 
\begin{equation}
	\label{eq: mu_m even}
	\mu_{m}(f(x) g(-i\nabla)) \leq \mu_{m/2 +1}(A_n) + \mu_{m/2 }(B_n) \leq \mu_{m/2 }(A_n) + \mu_{m/2 }(B_n),\quad m \textnormal{ even}.
\end{equation}
By estimating the norms $\|A_n\|_{\mathcal{L}^{2}}$ and $\|B_n\|_{\mathcal{L}^{1}}$ we obtain bounds on the singular values of $A_n$ and $B_n$ which will allow us to conclude. 
Since, $f_l$ and $g_k$ have disjoint supports, we first obtain that 
\begin{align*}
	\|A_n\|_{\mathcal{L}^{2}}^2 =\, &\mathrm{Tr} (A_n^*A_n) = \mathrm{Tr}\Big( \sum_{\substack{l+k\leq n\\l'+k'\leq n}} f_{l}(x) \overline{f_{l'}(x)}g_k(-i\nabla)\overline{g_{k'}(-i\nabla)}\,\Big)\\
	=&\, \mathrm{Tr}\Big( \sum_{l+k\leq n}\overline{g_k(-i\nabla)} |f_{l}(x)|^2 g_k(-i\nabla)\, \Big) .
\end{align*}
This expression can be computed thanks to the formula $\|f(x)g(-i\nabla)\|_{\mathcal{L}^2} = (2\pi)^{-d/2} \|f\|_{L^2} \|g\|_{L^2}$:
\begin{equation*}
	\|A_n\|_{\mathcal{L}^{2}}^2 
	 = \sum_{l+k\leq n}\|f_{l}(x) g_k(-i\nabla)\|_{\mathcal{L}^{2}}^2
	 =\,c_d
	 \sum_{l+k\leq n} \|f_l\|_{L^2}^2\|g_k\|_{L^2}^2
	 = \,c_d
	 \sum_{\substack{l+k\leq n\\ 2^{l-1}<a_{\mathbf{m}}\leq 2^l\\ 2^{k-1}<b_{\mathbf{p}}\leq 2^k}} a_{\mathbf{m}}^2b_{\mathbf{p}}^2\,.
 \end{equation*}
The number of $b_{\mathbf{p}}$ in the interval $(2^{k-1}, 2^{k}]$ is bounded by
\begin{align}
	\label{eq: number b_p}
	|\{\mathbf{p}: b_{\mathbf{p}} \geq 2^{k-1}\} | \leq 2^{-p'(k-1)} \|b_{\mathbf{p}}\|_{\ell^{p', \infty}}^* \leq 2^2 2^{-kp'}.
\end{align}
Using this in the norm of $A_n$ gives
\begin{multline*}
		\|A_n\|_{\mathcal{L}^{2}}^2 
		\lesssim_{\,d}   \sum_{\substack{l+k\leq n\\ 2^{l-1}<a_{\mathbf{m}}\leq 2^l}} a_{\mathbf{m}}^2 2^{2k}\sum_{2^{k-1}<b_{\mathbf{p}}\leq 2^k} 1\\
		\lesssim_{\,d}
	 \sum_{\substack{l \in \Z\\ 2^{l-1}<a_{\mathbf{m}}\leq 2^l}} a_{\mathbf{m}}^2 \sum_{k \leq n-l} 2^{2k-kp'}
		%=& \sum_{\substack{l \in \Z\\ 2^{l-1}<a_{\mathbf{m}}\leq 2^l}}a_{\mathbf{m}}^2 \ \frac{2^{(p'-2)(l-n)}}{1-2^{p'-2}}\\
		= \frac{2^{(2-p')n}}{1-2^{p'-2}}\sum_{\substack{l \in \Z\\ 2^{l-1}<a_{\mathbf{m}}\leq 2^l}} 2^{-(2-p')l}a_{\mathbf{m}}^2
		\leq  \frac{2^{(2-p')n}}{1-2^{p'-2}}\ , 
\end{multline*}
where in the last inequality we have used the bound
\begin{align*}
	\sum_{\substack{l \in \Z\\ 2^{l-1}<a_{\mathbf{m}}\leq 2^l}} 2^{-(2-p')l}a_{\mathbf{m}}^2 = & 	\sum_{\substack{l \in \Z\\ 2^{l-1}<a_{\mathbf{m}}\leq 2^l}} 2^{-(2-p')l}a_{\mathbf{m}}^{2-p'}a_{\mathbf{m}}^{p'} \leq \sum_{\mathbf{m}} a_{\mathbf{m}}^{p'} =1.
\end{align*}
By \cite[Theorem 4.5]{Simon05} we also have 
\begin{equation*}
	\|B_n\|_{\mathcal{L}^{1}}  \lesssim \sum_{\substack{l+k> n\\ 2^{l-1}<a_{\mathbf{m}}\leq 2^l\\ 2^{k-1}<b_{\mathbf{p}}\leq 2^k}} a_{\mathbf{m}}b_{\mathbf{p}}.
\end{equation*}
 Using again \eqref{eq: number b_p} we have 
\begin{align*}
	\|B_n\|_{\mathcal{L}^{1}} \lesssim & \sum_{\substack{l+k> n\\ 2^{l-1}<a_{\mathbf{m}}\leq 2^l}}a_{\mathbf{m}} 2^{k-kp'} = \sum_{\substack{l \in \Z\\ 2^{l-1}<a_{\mathbf{m}}\leq 2^l}}a_{\mathbf{m}} \sum_{k\geq n-l+1} 2^{(1-p')k}\\
	=& \sum_{\substack{l \in \Z\\ 2^{l-1}<a_{\mathbf{m}}\leq 2^l}}a_{\mathbf{m}} \frac{2^{(1-p')(n-l+1)}}{1-2^{1-p'}}\leq  \frac{2^{(1-p')n}}{1-2^{1-p'}},
\end{align*}
where we have used the following inequality
\begin{align*}
\sum_{\substack{l \in \Z\\ 2^{l-1}<a_{\mathbf{m}}\leq 2^l}}a_{\mathbf{m}} 2^{(1-p')(1-l)}=\sum_{\substack{l \in \Z\\ 2^{l-1}<a_{\mathbf{m}}\leq 2^l}}a_{\mathbf{m}}^{p'}	 a_{\mathbf{m}}^{1-p'} 2^{-(1-p')(l-1)} \leq \sum_{\mathbf{m}} a_{\mathbf{m}}^{p'} =1.
\end{align*}
Going back to \eqref{eq: mu_m uneven} and \eqref{eq: mu_m even}, it suffices to consider $m$ even. By the definition of the norms on the trace ideals $\mathcal{L}^{1}, \mathcal{L}^{2}$ and since the singular values are arranged in decreasing order, we have 
\begin{align*}
	\|B_n\|_{\mathcal{L}^{1}} \geq \sum_{k=1}^{m/2} \mu_k (B_n )\geq \frac{m}{2} \mu_{m/2} (B_n ),
\end{align*}
which implies 
\begin{align*}
	\mu_{m/2} (B_n ) \lesssim \frac{2}{m}\ \frac{2^{(1-p')n}}{1-2^{1-p'}} \lesssim \frac{2^{(1-p')n}}{m}.
\end{align*}
Analogously 
\begin{align*}
	\mu_{m/2} (A_n ) \lesssim \sqrt{\frac{1}{m}}\  \frac{2^{(1-p'/2)n}}{\sqrt{1-2^{p'-2}}},
\end{align*}
and hence 
\begin{align}
	\label{eq: ineq mu_m}
	\mu_m (f(x)g (-i\nabla)) \lesssim m^{-1} 2^{(1-p')n} + m^{-\frac12} \frac{2^{(1-p'/2)n}}{\sqrt{1-2^{p'-2}}}.% \leq \frac{2}{\min\{\sqrt{1-2^{p'-2}}, 1-2^{1-p'}\}}(m^{-1/2}2^{(1-p'/2)n} + m^{-1}2^{(1-p')n}).
\end{align}
Optimizing with respect to $n$ yields
\begin{align}
	\label{eq: ineq mu_mbis}
	\mu_m (f(x)g (-i\nabla)) \lesssim m^{-\frac1{p'}} (1-2^{p'-2})^{\frac{1}{p'}-1}\  \lesssim m^{-\frac1{p'}} (2-p')^{\frac{1}{p'}-1},
\end{align}
%
%
%\textcolor{red}{Now we look for $n$ such that 
%\begin{equation*}
%	\frac{m^{-1/2}2^{(1-p'/2)n}}{\sqrt{1-2^{p'-2}}} \leq c\, m^{-1/p'} (2-p')^{1/p'-1}
%\end{equation*}
%for some consant $c$ independent of $m$ and $p'$. The previous inequality implies 
%\begin{equation*}
%	2^{(1-p'/2)n} \leq c\, m^{1/2-1/p'} \frac{\sqrt{1-2^{p'-2}}}{(2-p')^{\frac{p'-1}{p'}}}\to c\sqrt{\ln 2} \quad \textnormal{as } p'\to 2.
%\end{equation*} 
%The left hand side is equal to one in the limit $p'\to 2$, hence choosing $c= 1/\sqrt{\ln 2}$ we can take any $n\in \Z$. 
%Analogously, we want to bound the second term in \eqref{eq: ineq mu_m} by the same quantity, hence we need $n$ to also satisfy 
%\begin{equation*}
%\frac{m^{-1}2^{(1-p')n}}{\sqrt{1-2^{p'-2}}}	\leq c m^{-1/p'} (2-p')^{1/p'-1}
%\end{equation*}
%which implies 
%\begin{equation*}
%	2^{(1-p')n} \leq c\, m^{1-1/p'} \frac{\sqrt{1-2^{p'-2}}}{(2-p')^{\frac{p'-1}{p'}}} \to c\,m^{1/2}\sqrt{\ln 2} \quad \textnormal{as } p'\to 2
%\end{equation*}
%while the left hand side now converges to $2^{-n}$ as $p'\to 2$. Hence choosing again $c= 1/\sqrt{\ln 2}$ we can take any $n\in \N$. 
%}
%
%Now, choose $q=q(m)$ such that $2^{q-1}\leq m^{1/p'} \leq 2^{q}$ and fix $n=n(m)<0$ such that $2^{q-1}\leq 2^{-n} \leq 2^{q}$, which implies $2^n \leq 2^{-q}\leq m^{-1/p'}$. We have
%\begin{align*}
%		\mu_m (f(x)g (-i\nabla)) \lesssim \frac{2}{\min\{\sqrt{1-2^{p'-2}}, 1-2^{1-p'}\}}  m^{-1/p'},
%\end{align*}
which proves the statement of the theorem.
\end{proof}

We conclude with the proof of Lemma \ref{lm:embedding}.

%\begin{lemma}
%	Let $d\geq 1$, $1<p'\leq2$ and $r>d(\frac{1}{p'}-\frac{1}{2})$. Then the constant in the embedding $L^2(\langle x\rangle^{2r}\mathrm{d}x) \hookrightarrow \ell^{p'}(L^2)$ is uniform in the limit $p'\to 2$.
%\end{lemma}

\begin{proof}[Proof of Lemma \ref{lm:embedding}]
	Let $q$ be defined by $\frac{1}{q} + \frac{1}{2} = \frac{1}{p'}$. We have
	\begin{align*}
		\|f\|_{\ell^{p'}(L^2)} = & \|\,\|\chi_{\mathbf{m}}f\|_{L^2} \|_{\ell^{p'}} \lesssim \| \langle \mathbf{m}\rangle ^{-r} \|\chi_{\mathbf{m}}f\|_{L^2(\langle x\rangle^{2r}\mathrm{d}x)} \|_{\ell^{p'}} \lesssim \|\langle \mathbf{m}\rangle ^{-r}\|_{\ell^q}\|f\|_{L^2(\langle x\rangle^{2r}\mathrm{d}x)},
	\end{align*}
where we choose $r$ such that $rq>d$, so that $\langle\mathbf{m}\rangle^{-r}$ indeed belongs to $\ell^q(\Z^d)$. By straightforward computations one obtains 
\begin{align*}
	\|f\|_{\ell^{p'}(L^2)} \lesssim_{\,d} &\, \frac{(rq-d + 1)^{1/q}}{(rq-d)^{1/q}} \|f\|_{L^2(\langle x\rangle^{2r}\mathrm{d}x)}.
\end{align*}
It then suffices to observe that, given $0<\delta<1$, the constant appearing in the right hand side of the previous inequality is uniformly bounded in $p'\in[2-\delta,2)$.
\end{proof}

\bibliographystyle{plain}
\bibliography{biblioCLR}

\end{document}